\documentclass[12pt,a4paper]{article}
\usepackage[margin=1.5in]{geometry}
\usepackage{graphicx,epsfig}
\usepackage{amsfonts}
\usepackage{amsthm, amsmath}
\usepackage{xcolor}
\usepackage{textcomp}
\usepackage{listings}
\usepackage{multirow}
\usepackage{adjustbox,lipsum}
\usepackage{dsfont}
\usepackage{xfrac}
\usepackage{nicefrac}
\usepackage{stmaryrd}
\usepackage{mathtools}
\usepackage{bm}

\usepackage{amssymb}

\usepackage{subfigure}
\usepackage{subfloat}

\usepackage[symbol]{footmisc}
\renewcommand{\thefootnote}{\fnsymbol{footnote}}

\usepackage{fancyhdr}
\usepackage{tikz-cd}

\usepackage{enumerate}
\usepackage{enumitem}
\usepackage{float}

\usepackage[utf8]{inputenc}
\usepackage[T1]{fontenc}

\newtheorem{theorem}{Theorem}[section]

\newtheorem{corollary}[theorem]{Corollary}
\newtheorem{lemma}[theorem]{Lemma}

\newtheorem{example}[theorem]{Example}

\newtheorem{remark}[theorem]{Remark}
\theoremstyle{definition}
\newtheorem{definition}[theorem]{Definition}

\newtheoremstyle{case}{}{}{}{}{}{:}{ }{}
\theoremstyle{case}

\usepackage[hidelinks]{hyperref}
\hypersetup{
    colorlinks,
    citecolor=black,
    filecolor=black,
    linkcolor=black,
    urlcolor=black
}

\usepackage[backend=bibtex,style=alphabetic]{biblatex}
\DeclareMathOperator*{\Plim}{\mathbb{P}\text{-}lim}

\providecommand{\keywords}[1]
{
	\small	
	\textbf{\textit{Keywords---}} #1
}

\begin{document}

\title{It\^o Stochastic Differentials}
\author{
	John Armstrong\footnotemark[2] \and Andrei Ionescu\footnotemark[2]}
\date{}
\maketitle

\begin{NoHyper}
\footnotetext{\footnotemark[2]King's College London}
\end{NoHyper}

\begin{abstract}
	We give an infinitesimal meaning to the symbol $dX_t$ for a continuous semimartingale $X$ at an instant in time $t$. We define a vector space structure on the space of differentials at time $t$ and deduce key properties consistent with the classical It\^o integration theory. In particular, we link our notion of a differential with It\^o integration via a stochastic version of the Fundamental Theorem of Calculus. Our differentials obey a version of the chain rule, which is a local version of It\^o's lemma. We apply our results to financial mathematics to give a theory of portfolios at an instant in time. 
\end{abstract}



\renewcommand{\thefootnote}{\arabic{footnote}}

\section{Introduction} \label{section: introduction}
It\^o first defined a stochastic integral with respect to Brownian motion in 1944 \cite{ito1944}. This integral has since been extended to include a larger class of processes as integrators. Semimartingales now form the largest class of processes with respect to which the It\^o integral is defined (see for instance \cite{legall2016martingales,protter2005stochasticintegration}). The resulting body of research that uses the notion of this stochastic integral is called It\^o calculus.

Symbolically however, the majority of It\^o calculus uses \emph{differential} notation to study stochastic \emph{integral} equations. This convention has been so convenient in fact that it has become commonplace to simply call these \emph{stochastic differential equations} (SDEs). We do this and yet, there are no proper definitions of differentials in It\^o calculus. The absence of such a definition certainly did not hinder the development of It\^o calculus which has far-reaching applications to the fields of probability theory, stochastic analysis, differential geometry and mathematical finance. Nevertheless, the very same mathematicians that developed the notion of the It\^o integral to what it has become today have also pondered the existence of its differential counterpart. Laurent Schwartz admitted in \cite{schwartz1986grosProduitsTensoriels} that there is nothing "\emph{ponctuel}" (instantaneous) in the object we call $dX_t$. Schwartz developed a theory of SDEs on manifolds with the belief that the differential would one day be defined rigorously. This is the theory of second order tangent vectors and Schwartz morphisms \cite{emery1989,emery2007}. Michel Emery, who himself had been Paul-Andr{\'e} Meyer’s doctoral student, stated in \cite{emery1989} that the “existence of the [stochastic differential] is metaphysical and one is free not to believe in it.” 

In our work we present an explicit definition of the It\^o stochastic differential. Its existence is no longer metaphysical and one is now bound to believe it in as much as one believes in the It\^o integral. We denote our differential as $d^p(X)_t$ to distinguish it from the usual notation. We define this differential for a large subspace of continuous semimartingales, including It\^o processes, and do so in a local and intrinsic way. This differential is completely compatible with the theory of It\^o integration in much the same way that differentiation and (Riemann) integration are related via the Fundamental Theorem of Calculus. In fact, we prove a stochastic version of the FTC.

Naturally, one may wonder if an explicit definition of the differential has a consequence on the theory of It\^o calculus beyond just providing a sense of completeness. After all, the differential has been used \emph{implicitly} via the definition of It\^o integration. We demonstrate that our differential has applications which its integral counterpart cannot offer on its own. This is precisely because $d^p(X)_t$ has properties one should come to expect from a differential, one of which is that it must be defined locally in time. The It\^o integral, by its very nature, cannot be locally defined in time. Using a well-defined differential, we can now interpret SDEs at a single point in time, which we could not do using It\^o integration. Moreover, we define $d^p(X )_t$ intrinsically for $\mathbb{R}^n$-valued processes.

We outline financial applications using the probabilistic and \lq forward-looking\rq\space  perspective of our theory. More precisely, we give a rigorous interpretation of continuous-time financial concepts at an moment in time. For instance, the value of a portfolio $\{\Pi_t\}_{t\geq 0}$ will be \lq instantaneously hedged\rq\space at time $t$ if and only if it has the same dynamics as the risk-free bank account at that time. Equivalently in the language of differentials, $\{\Pi_t\}_{t\geq 0}$ must satisfy $d\Pi_t = r\Pi_t dt$. It is important to understand that this statement requires the definition of a differential. Then we define the concept of an instantaneous portfolio in a complete market solely using invariants of the latter, without the need to explicitly mention assets. We show how this framework gives rise to infinitesimal market portfolio theory and prove a one-mutual fund theorem.

To the extent of our knowledge, this is the first time a stochastic differential at an instant in time has been rigorously defined. There are however several notions of derivatives in stochastic analysis. The Malliavin calculus \cite{nualart2006malliavinCalculus} is used to take derivatives of random variables with respect to paths in the classical Wiener space and is used to give an explicit form to the $W$-derivative of a random variable in the Martingale Representation theorem. Likewise, Allouba and Fontes \cite{allouba2006differentiationtheory} provide a theory of pathwise derivatives of semimartingales with respect to Brownian motion. The main result of their work is a stochastic version of the Fundamental Theorem of Calculus however this is a derivative with respect to a Brownian motion and not with respect to time. In the theory of rough paths, one removes the probability measure from stochastic integration by focusing instead on the regularity properties of typical paths. The Gubinelli derivative \cite{gubinelli2003controlling} essentially allows one to find a first order Taylor approximation of $\alpha$-H\"older paths with respect to other $\alpha$-H\"older paths. This is close in spirit to what we have achieved but differs as we work in the probabilistic setting. Functional Itô calculus is a theory developed by Dupire \cite{dupire2009functionalitocalculus}, and later Cont and Fournié \cite{cont2013functionalitocalculus}, that extends It\^o calculus to functionals of Itô processes. Functional It\^o calculus introduces the Dupire horizontal and vertical derivatives of functionals which respectively act as time and space derivatives. It has links with Malliavin calculus and rough path theory (see \cite{cont2019pathwiseIntegration}). Functionals of paths consider the whole history of paths and as such are not a local notion which is why we do not consider them here. Differentiation techniques in Malliavin calculus have successfully been applied to stochastic differential equations driven by fractional noise in \cite{NUALART2009391}. We believe that the differential approach in this paper could adequately be adapted to also study stochastic dynamics with fractional noise.

We define $d^p(X)_t$ using convergence in probability. We also consider two alternative definitions for the stochastic differential: the first using convergence in expectation, and the second using almost-sure convergence. However, the results in both the latter approaches are less satisfactory and our definition using convergence in probability yields all the properties one hopes for from a stochastic differential. We still include the expectation and almost-sure convergence approaches here to compare them to the convergence in probability approach. The expectation approach is similar to Nelson's idea of a `mean-forward derivative' in \cite{nelson1967bm}.

This paper is set out as follows: after some preliminary definitions in Section \ref{section: preliminaries}, in Section \ref{section: stochastic differentials using convergence in probability} we introduce the infinitesimal differential $d^p(X)_t$ on $\mathbb{R}^n$ using convergence in probability and we state the main results. In particular we show that any stochastic differential equation, in the classical sense of the term, corresponds naturally with its analogue written using the $d^p(\cdot)$ formalism. In Section \ref{section: financial applications} we describe financial applications in continuous-time complete markets. In Section \ref{section: stochastic differentials using convergence in mean} we introduce a version of the differential using convergence in mean. The relationships with $d^p(X)_t$ are outlined and analogous results to Section \ref{section: stochastic differentials using convergence in probability} are presented. Finally in Section \ref{section: stochastic differentials using almost sure convergence} we give a third definition of the differential, $d^\text{a.s.}(X)_t$ using almost sure convergence and an analogue to the Lebesgue differentiation theorem is proved.

To clarify our contribution, we provide a reference to all definitions and results we use from other sources and remark where we could not find a proof in the literature, but do not believe the result is new. If we simply state a definition or a result, we believe it is our own.

\section{Preliminaries}\label{section: preliminaries}
\subsection{Asymptotic notation in probability}

Just as in deterministic calculus, we consider the limits of difference quotients to define $dX_t$. These limits are taken in probability and we now introduce some definitions relating asymptotic notation (also commonly known as ``little $o$ and big $\mathcal{O}$'' notation) to convergence in probability. All the definitions and results of this section, unless followed by a proof, are taken from \cite{bishop2007multivariateAnalysis} and \cite{kallenberg2001foundationsmodernprobability}.

\begin{definition}[Convergence in probability \cite{bishop2007multivariateAnalysis}]
Let $\{X_n\}_{n\in\mathds{N}}$ be a sequence of random variables. One writes $X_n = o_p(1_n)$ if $\{X_n\}_{n\in\mathds{N}}$ converges to $0$ in probability. If $\{a_n\}_{n\in\mathds{N}}$ is a sequence of real constants, ones writes $X_n = o_p(a_n) \Longleftrightarrow \frac{X_n}{a_n} =o_p(1_n)$.
\end{definition}

\begin{definition}[Stochastic boundedness \cite{bishop2007multivariateAnalysis}]
$\{X_n\}_{n\in\mathds{N}}$ is said to be \emph{bounded in probability} (or \emph{uniformly tight}) if $\forall \epsilon>0$, $\exists K>0$ and $\exists N\in\mathds{N}$ such that
\begin{equation*}
\mathds{P}\big(|X_n|>K\big)<\epsilon,~\forall n\geq N.
\end{equation*}
If this is the case, one writes $X_n = \mathcal{O}_p(1_n)$. If $\{a_n\}_{n\in\mathds{N}}$ is a sequence of real constants, ones writes $X_n = \mathcal{O}_p(a_n) \iff \frac{X_n}{a_n} = \mathcal{O}_p(1_n)$.
\end{definition}

We adapt the previous two definitions to continuous-time processe so that we can apply them to define our stochastic differentials later.
\begin{definition}[Asymptotic notation in probability for continuous-time processes]
	Let $\{X_t\}_{t\geq 0}$ be a continuous-time stochastic process and $f:\mathbb{R}_{\geq 0} \rightarrow \mathbb{R}_{\geq 0}$ be an increasing function. We write $X_h = o_p\left(f(h)\right)$ if for all $\epsilon,\delta >0$, there exists $\eta>0$ such that
	\begin{equation*}
		\mathds{P}\left(\left\lvert\frac{X_h}{f(h)}\right\rvert>\epsilon\right) < \delta,
	\end{equation*}
	for all $h\in(0,\eta)$. Similarly, we write that $X_h = \mathcal{O}_p\left(f(h)\right)$ if for all $\epsilon>0$, there exists $K>0$ and $\eta>0$ such that 
	\begin{equation*}
		\mathds{P}\left(\left\lvert\frac{X_h}{f(h)}\right\rvert>K\right)<\epsilon,
	\end{equation*}
	for all $h\in(0,\eta)$. If $\{X_t\}_{t\geq 0}$ depends on some other parameters then we do not require uniformity.
\end{definition}

\begin{lemma}\label{lemma: properties of little and big o nonation in probability} $o_p(\cdot)$ and $\mathcal{O}_p(\cdot)$ obey the same properties as do $o(\cdot)$ and $\mathcal{O}(\cdot)$. For sequences of reals $\{a_n\}$ and $\{b_n\}$,
	\begin{align*}
	o_p(a_n)+o_p(b_n) &= o_p\left(\max\{a_n,b_n\}\right)         &  o_p(1_n)+\mathcal{O}_p(1_n) &=\mathcal{O}_p(1_n)               \\
	o_p(a_n) o_p(b_n)&=o_p(a_nb_n)      &  o_p(a_n) \mathcal{O}_p(b_n) &=  o_p(a_nb_n)     
	\end{align*}
\end{lemma}

The following facts will be used throughout proofs of our main results so we state them here and refer to them when needed.

\begin{lemma}
	Any $\mathds{R}^d$ valued random variable is bounded in probability. That is, if $Z:\Omega \rightarrow \mathds{R}^d$ is a random variable, then $Z = \mathcal{O}_p(1_n)$.
\end{lemma}

\begin{remark}
	Most of our definitions deal with increments of processes so to lighten the notation, we set $H_{t,s} \coloneqq H_s -H_t$ for any times $0\leq t\leq s$.
\end{remark}

\begin{lemma}\label{lemma: sup of continuous process is op1}
	If $H:[0,T]\times\Omega\rightarrow \mathds{R}^d$ is a stochastic process with sample paths that are continuous with probability $1$ at $t\in[0,T]$, then $\sup_{s\in(t,t+h)}\lVert H_{t,s}\rVert = o_p(1_h)$ and $\sup_{s\in(t,t+h)}\lVert H_s\rVert = \mathcal{O}_p(1_h)$.
\end{lemma}
Note that we have not specified the norm here. As all norms are equivalent on $\mathbb{R}^d$, this result holds for any arbitrary choice of norm. 
\begin{proof}
	By hypothesis we have
	\begin{equation*}
	\mathds{P}\Big(\big\{\omega\in\Omega: \lim\limits_{s\rightarrow t}\lVert H_{t,s}(\omega)\rVert=0 \big\} \Big)=1.
	\end{equation*}
	This implies that $\sup_{s\in(t,t+h)}\Vert H_{t,s}\rVert$ converges to $0$ on almost all sample paths as $h\rightarrow 0$. As almost sure convergence implies convergence in probability, $\sup_{s\in(t,t+h)}\lVert H_{t,s}\rVert=o_p(1_h)$. Then, since $H_t=  \mathcal{O}_p(1_h)$,
	\begin{equation*}
	\sup_{s\in(t,t+h)}\lVert H_s\rVert\leq \sup_{s\in(t,t+h)}\lVert H_{t,s}\rVert+\lVert H_t\rVert = o_p(1_h)+\mathcal{O}_p(1_h)=\mathcal{O}_p(1_h).
	\end{equation*}
\end{proof}

\section{It\^o stochastic differentials using convergence in probability} \label{section: stochastic differentials using convergence in probability}
\subsection{Motivation: visualisation of a differential for diffusion processes}
Ordinary differential equations may be written as integral equations, but the differential form is often more intuitively appealing as one can easily visualise a tangent vector. To visualise the differential of a process, consider 
\begin{equation}
	X_t = W_t^2 - t +t^2 \label{eq: numerical example},
\end{equation} 
which satisfies 
\begin{equation*}
	dX_t = 2W_tdW_t + 2tdt.
\end{equation*}
Since the coefficients of $dW_t$ and of $dt$ vanish at time $0$, we expect the differential $dX_t$ to vanish at time $0$.

In Figure \ref{fig: process} we plot a sample path (jagged blue), the mean (black) and a fan diagram (dotted green) showing the $5^\text{th}$ and $95^\text{th}$ percentiles of (\ref{eq: numerical example}). As a contrast, we plot similar features for the process $W_t$ in Figure \ref{fig: brownian motion} and for $W_t^2$ in Figure \ref{fig: squared brownian motion} which should not have a differential zero.

\begin{figure}[H]
	\centering
	\subfigure[\label{fig: process}$X_t = W_t^2 -t +t^2$]{\includegraphics[width=0.32\textwidth]{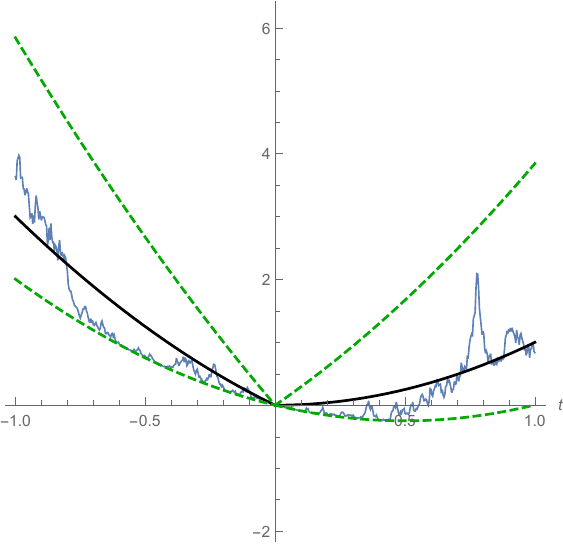}} 
	\subfigure[\label{fig: brownian motion}$W_t$]{\includegraphics[width=0.32\textwidth]{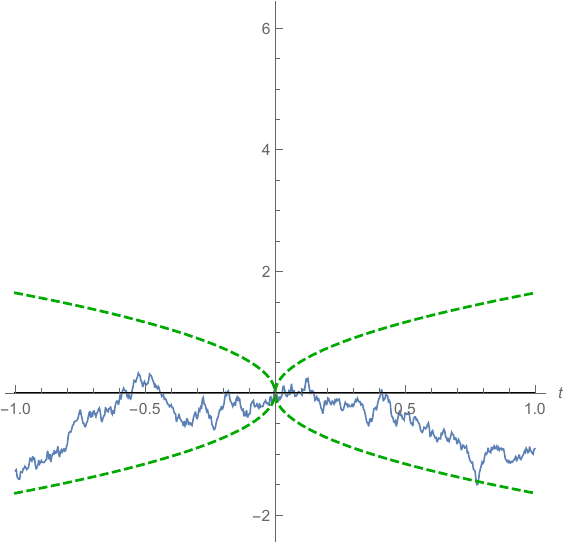}} 
	\subfigure[\label{fig: squared brownian motion}$W^2_t$]{\includegraphics[width=0.32\textwidth]{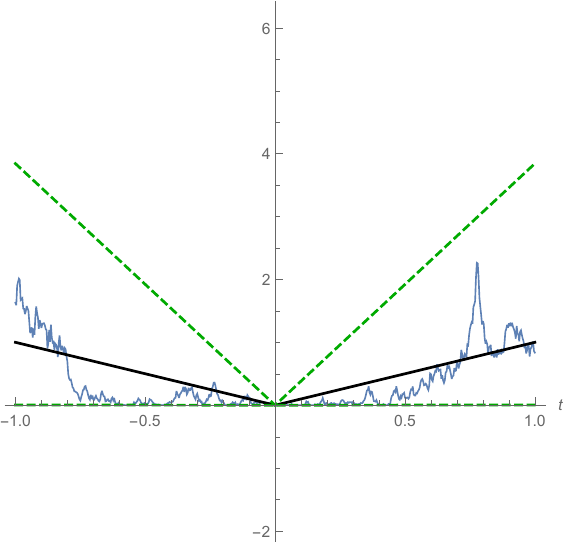}}
	\caption{Sample path, mean process, and $5^\text{th}$ and $95^\text{th}$ percentiles of three processes}
	\label{fig: all plots}
\end{figure}

The percentile curves in Figure \ref{fig: process} are not tangent to the $y$-axis, unlike the corresponding curves in Figure \ref{fig: brownian motion}; this shows that the martingale part of $X_t$ is in some sense small. The mean curve is tangent to the $x$-axis unlike the corresponding curve in Figure \ref{fig: squared brownian motion}; this shows that the finite variation part of $X_t$ is also in some sense small. By Theorem $3$ in \cite{armstrong2016coordinatefree}, these features ensure the coefficients of $dW_t$ and $dt$ in the SDE for $X_t$ vanish at time $0$.

Our aim is to develop a definition of a differential which can apply to more general processes than diffusion processes, but still admits this visual interpretation when applied to the latter. This requires us to precisely define what we mean by \lq in some sense small\rq.

We note the clear time asymmetry: having a differential of $0$ indicates that the change in $X_t$ is small for small positive $t$, but says nothing about negative values of $t$. In summary, the part of a fan diagram for the process in the direction of increasing time allows us to visually identify the differential.

\subsection{Chords and variations bounded in probability}

For a real-valued continuous-time stochastic process $\{X_t\}_{t\geq 0}$, and for any two times $0\leq t < s$, we denote the running maximum starting at $t$ by
\begin{equation*}
    X^*_{t,s} \coloneqq \sup_{u\in(t,s)}\lvert X_{t,u} \rvert \text{ and }X^*_s \coloneqq X^*_{0,s}.
\end{equation*}
For a $\mathbb{R}^m$-valued process, $X^*_{t,s}\coloneqq \max\limits_{1\leq i\leq m}\left(X^{(i)}\right)^*_{t,s}$.

\begin{definition}[Chords bounded in probability] \label{def: chords bounded in probability}
	Let $X$ be a continuous semimartingale with canonical decomposition $X=X_0+A+M$. Furthermore, let $A=A^1-A^2$ be the canonical decomposition of the finite variation part as a difference of non-decreasing processes. Then we say that $X$ has \emph{chords bounded in probability} (CBP) at time $t$ if
	\begin{equation*}
		A^1_{t,t+h} = \mathcal{O}_p(h),~A^2_{t,t+h}=\mathcal{O}_p(h) \text{ and } \left(M^*_{t,t+h}\right)^2 = \mathcal{O}_p(h).
	\end{equation*} 
We denote the set of all real-valued continuous semimartingales with CBP at time $t$ by $\mathcal{S}^\text{CBP}_t(\mathbb{R})$.
\end{definition}

The set of all processes with CBP at time $t$ is the space of processes for which we will define a differential. We avoid calling these processes `differentiable' at $t$. This is because in the deterministic case, having chords bounded in probability does not imply the differentiability of a function although the converse is true. For instance, the absolute value $\lvert t\rvert$ has chords bounded in probability by $1$ but is not differentiable at $t = 0$. 

\begin{definition}[Total variation and quadratic (co)variation \cite{legall2016martingales,protter2005stochasticintegration}]
Let $X$ be a continuous $\mathbb{R}^m$-valued semimartingale with canonical decomposition $X = X_0 + A + M$. The total variation of $A$ over $[t,t+h]$ is defined as 
\begin{equation*}
	TV(A)_{t,t+h} \coloneqq \Plim\limits_{\lVert \pi\rVert\rightarrow 0}\sum_{(t_i,t_{i+1})\in\pi}\lVert A_{t_i,t_{i+1}}\rVert_2,
\end{equation*}
and the quadratic variation of $M$ as
\begin{equation*}
	\lbrack M\rbrack_{t,t+h} \coloneqq \Plim\limits_{\lVert \pi\rVert\rightarrow 0}\sum_{(t_i,t_{i+1})\in\pi}\lVert M_{t_i,t_{i+1}}\rVert_2^2,
\end{equation*}
where the limit is taken in probability over all finite partitions of $[t,t+h]$ and where $\lVert\cdot \rVert_2$ is the standard Euclidian norm on $\mathbb{R}^m$. The bracket ignores finite variation processes so that $\lbrack X\rbrack_t = \lbrack M\rbrack_t$ for all $t\geq 0$. The quadratic covariation of two continuous semimartingales $X$ and $Y$ is defined using the polarisation identity $\lbrack X,Y\rbrack \coloneqq \frac{1}{2}\left(\lbrack X+Y\rbrack - \lbrack X\rbrack - \lbrack Y\rbrack\right)$.
\end{definition}

Since all norms on $\mathbb{R}^m$ are equivalent, the choice of the Euclidian norm above turns out to be arbitrary. Our results would still hold if we had chosen any other norm on $\mathbb{R}^m$. We refer to Lemma \ref{lemma: tv and qv definitions do not depend on choice of norm} in the Appendix for the formal justification of this statement.  

\begin{definition}[Variations bounded in probability]
We say that a continuous semimartingale $X$ with canonical decomposition $X=X_0+A+M$ has \emph{variations bounded in probability} (VBP) at time $t$ if $TV(A)_{t,t+h} = \mathcal{O}_p(h)$ and $\lbrack M\rbrack_{t,t+h} = \mathcal{O}_p(h)$. We denote the set of all real-valued continuous semimartingales with VBP at time $t$ by $\mathcal{S}^\text{VBP}_t(\mathbb{R})$.
\end{definition} 

\begin{remark}
	For the remainder of this section, and unless clearly stated by the context of the results, we consider $\mathbb{R}$-valued continuous semimartingales. The corresponding results for $\mathbb{R}^m$-valued processes are a consequence of the one-dimensional results applied component-wise and the equivalence of $p$-norms in Euclidean spaces. For a rigorous justification, we refer to Lemma \ref{lemma: all results are true iff they hold component wise}.
\end{remark}

The notion of VBP is equivalent to that of CBP for a general continuous semimartingale. We introduce both definitions for two reasons: the first is that the definition of CBP does not make any mention concept of variation (whether total or quadratic). The definition of variation makes use of the limit of Riemann-like sums which are analogous to integrals. We want to avoid using integral-like objects to define differential so as not to create tautological results relating our stochastic differentials and the usual stochastic integrals. On the other hand, we do introduce the definition of VBP as a technical tool. Using the properties of total and quadratic variation in our proofs makes these shorter and more direct.

We now show the equivalence between the two definitions. It is straightforward to prove for the finite variation part of the process. 
\begin{lemma}
	Let $A$ be a continuous finite variation process with $A_0 = 0$. Then $A\in\mathcal{S}^\text{CBP}_t \iff A\in\mathcal{S}^\text{VBP}_t$.
\end{lemma}

\begin{proof}
The unique splitting of $A$ as the difference of two continuous non-decreasing processes beginning at zero is 
	\begin{equation*}
		A_t = \underbrace{\frac{1}{2}\left(TV(A)_t+A_t\right)}_{A^1_t} - \underbrace{\frac{1}{2}\left(TV(A)_t-A_t\right)}_{A^2_t}.
	\end{equation*} 
Suppose $A^1_{t,t+h} = \mathcal{O}_p(h)$ and $A^2_{t,t+h} = \mathcal{O}_p(h)$. Then $TV(A)_{t,t+h} = A^1_{t,t+h} + A^2_{t,t+h} = \mathcal{O}_p(h)\Rightarrow A\in\mathcal{S}^\text{VBP}_t$. Conversely, if $TV(A)_{t,t+h}=\mathcal{O}_p(h)$ then $A_{t,t+h} = \mathcal{O}_p(h)$ since $\lvert A_{t,t+h}\rvert \leq TV(A)_{t,t+h}$. Hence $\frac{1}{2}\left(TV(A)_{t,t+h}\pm A_{t,t+h}\right) = \mathcal{O}_p(h) \Rightarrow A\in\mathcal{S}^\text{CBP}_t$.
\end{proof}

To show the equivalence for local martingales we need a preliminary lemma relating the joint density of their running supremum and their quadratic variation. The Burkholder--Davis--Gundy inequality is the classical result which provides this sort of relation and in our case, we only need a piece of the proof of the BDG inequality, namely `good $\lambda$ inequalities'.

\begin{definition}[Good $\lambda$ inequality \cite{cohen2015}]
	The pair of non-negative random variables $(X,Y)$ is said to satisfy a \emph{good $\lambda$ inequality} for a constant $\beta>1$ and $\psi:\mathbb{R}_+\rightarrow\mathbb{R}_+$ satisfying $\psi(\delta)\rightarrow 0$ as $\delta\rightarrow 0$ such that
	\begin{equation*}
		\mathds{P}\left(X>\beta\lambda, Y<\delta\lambda\right) \leq \psi(\delta)\mathds{P}\left(X\geq \lambda\right),~\forall\delta,\lambda>0.
	\end{equation*}
\end{definition}

\begin{lemma}
	Let $M$ be a continuous local martingale with $M_0 = 0$. Then
	\begin{enumerate}[label=(\roman*), itemsep=0pt, topsep=0pt]
		\item $\left(M^*_\infty,\lbrack M\rbrack_\infty^{\nicefrac{1}{2}}\right)$ satisfies a good $\lambda$ inequality for any $\beta>1$ with $\psi(\delta) = \frac{\delta^2}{(\beta-1)^2}$, and
		\item $\left(\lbrack M\rbrack_\infty^{\nicefrac{1}{2}},M^*_\infty\right)$ satisfies a good $\lambda$ inequality for any $\beta>1$ with $\psi(\delta) = \frac{4\delta^2}{\beta^2-1}$.
	\end{enumerate}
\end{lemma}

\begin{proof}
	See Lemma $11.5.3.$ of \cite{cohen2015}.
\end{proof}

\begin{corollary}\label{cor: good lambda inequality}
	Let $N$ be a continous local martingale. Then for any $t\geq  0$ and $h>0$,
	\begin{enumerate}[label=(\roman*), itemsep=0pt, topsep=0pt]
		\item $\left(N^*_{t,t+h},\lbrack N\rbrack_{t,t+h}^{\nicefrac{1}{2}}\right)$ satisfies a good $\lambda$ inequality for any $\beta>1$ with $\psi(\delta) = \frac{\delta^2}{(\beta-1)^2}$, and
		\item $\left(\lbrack N\rbrack_{t,t+h}^{\nicefrac{1}{2}},N^*_{t,t+h}\right)$ satisfies a good $\lambda$ inequality for any $\beta>1$ with $\psi(\delta) = \frac{4\delta^2}{\beta^2-1}$.
	\end{enumerate}
\end{corollary}
\begin{proof}
	We could not find a proof of this result in the literature, but we do not claim it is new. The proof is a simple application of the previous result so we prove it here. The process $M_s=N_s - N_t$, $s\geq t$ is a martingale with respect to the filtration $\left\{\mathcal{F}_s:s\geq t\right\}$ and it `starts' at $M_t=0$ and so does $\left\{\lbrack N\rbrack_{t,s}:s\geq t\right\}$. We stop the martingale $N$ in the lemma above at $\tau = t+h$ and use that $\lbrack N^\tau\rbrack = \lbrack N \rbrack^\tau$ for any stopping time $\tau$.
\end{proof}

\begin{lemma}\label{lemma: CBP is equivalent to vbp for local martingales}
	Let $M$ be a continuous local martingale. Then
	\begin{equation*}
		(M^*_{t,t+h})^2 = \mathcal{O}_p(h) \iff \lbrack M\rbrack_{t,t+h} = \mathcal{O}_p(h).
	\end{equation*}
That is, $$M\in\mathcal{S}^\text{CBP}_t \iff M\in\mathcal{S}^\text{VBP}_t.$$
\end{lemma}

\begin{proof}
	For the `$\Rightarrow$' part, let $\lambda = \sqrt{Kh}$ and $\delta = 1$ in Corollary \ref{cor: good lambda inequality}, then
	\begin{equation*}
		\mathds{P}\left(\lbrack M\rbrack_{t,t+h}>\beta^2Kh, \left(M^*_{t,t+h}\right)^2<Kh\right) \leq \frac{4}{\beta^2-1}\mathds{P}\left(\lbrack M\rbrack_{t,t+h}\geq Kh\right),
	\end{equation*}
for any $\beta > 1$ and $K,h>0$. Hence,
\begin{multline*}
	\mathds{P}\left(\lbrack M\rbrack_{t,t+h}>\beta^2 Kh\right) 
	 = \mathds{P}\left(\lbrack M\rbrack_{t,t+h}>\beta^2 Kh , \left(M^*_{t,t+h}\right)^2\geq Kh \right) \\
	 + \mathds{P}\left(\lbrack M\rbrack_{t,t+h}>\beta^2 Kh, \left(M^*_{t,t+h}\right)^2 < Kh \right) \leq \mathds{P}\left(\left(M^*_{t,t+h}\right)^2\geq Kh\right) + \frac{4}{\beta^2-1}.
\end{multline*}
For any $\epsilon>0$, we can pick $K>0$ and $\beta>1$ large enough and $\eta >0$ such that
\begin{equation*}
	\mathds{P}\left(\left(M^*_{t,t+h}\right)^2 \geq Kh \right) < \frac{\epsilon}{2},~\forall h\in(0,\eta) \text{ and }\frac{4}{\beta^2-1} < \frac{\epsilon}{2}.
\end{equation*}
Hence,
\begin{equation*}
	\mathds{P}\left(\lbrack M\rbrack_{t,t+h} > \beta^2Kh\right) < \epsilon,~\forall h\in(0,\eta).
\end{equation*}
For the `$\Leftarrow$' part, we proceed similarly. Using Corollary \ref{cor: good lambda inequality} we also have 
\begin{equation*}
	\mathds{P}\left(\left(M^*_{t,t+h}\right)^2>\beta^2 Kh, \lbrack M\rbrack_{t,t+h} < Kh \right) \leq \frac{1}{(\beta-1)^2}\mathds{P}\left(\left(M^*_{t,t+h}\right)^2> Kh\right),
\end{equation*}
for all $\beta >1$. Hence,
\begin{equation*}
	\mathds{P}\left(\left(M^*_{t,t+h}\right)^2>\beta^2Kh\right) \leq \frac{1}{(\beta-1)^2} + \mathds{P}\left(\lbrack M\rbrack_{t,t+h}\geq Kh\right).
\end{equation*}
For any $\epsilon>0$, we can pick $K>0$ and $\eta>0$ such that
\begin{equation*}
	\mathds{P}\left(\lbrack M\rbrack_{t,t+h}\geq Kh\right) <\frac{\epsilon}{2},~\forall h\in(0,\eta).
\end{equation*}
We can pick $\beta >1$ large enough such that $\frac{1}{(\beta-1)^2}<\frac{\epsilon}{2}$. Thus
\begin{equation*}
	\mathds{P}\left(\left(M^*_{t,t+h}\right)^2 >\beta^2Kh\right) <\epsilon,~\forall h\in(0,\eta).
\end{equation*}
\end{proof}

\begin{corollary} \label{lemma: VBP is equivalent to CBP}
	$$\mathcal{S}^\text{CBP}_t = \mathcal{S}^\text{VBP}_t$$
\end{corollary}

\begin{lemma}
	 $\mathcal{S}^\text{CBP}_t$is a vector subspace of the space of continuous semimartingales for any $t\geq 0$.
\end{lemma}

\begin{proof}
	If $X\in\mathcal{S}^\text{CBP}_t$, it is easy to check that $a X\in\mathcal{S}^\text{CBP}_t$ for any $a\in\mathbb{R}$. It remains to show that the sets are closed under addition. Let $X$ and $Y\in\mathcal{S}^\text{VBP}_t$ with canonical decompositions $X= X_0+A+M$ and $Y=Y_0+B+N$. Naturally, their sum has canonical decomposition $X+Y=(X_0+Y_0)+(A+B)+(M+N)$. Let $\epsilon>0$. By definition, we can find some $K>0$ large enough and some $\eta>0$ such that
	\begin{equation*}
		\min\left\{\mathbb{P}\left(TV(A)_{t,t+h}>\frac{Kh}{2}\right),\mathbb{P}\left(TV(B)_{t,t+h}>\frac{Kh}{2}\right)\right\} < \frac{\epsilon}{2},~\forall h\in(0,\eta).
	\end{equation*}
The total variation is sub-additive so we deduce that
\begin{equation*}
	\mathbb{P}\left(TV(A+B)_{t,t+h}>Kh\right) < \epsilon,~\forall h\in(0,\eta).
\end{equation*}
As $\epsilon$ was arbitrarily chosen, this implies that $TV(A+B)_{t,t+h} = \mathcal{O}_p(h)$. One can check that $2\lbrack X+Y\rbrack_{t,t+h} \leq \lbrack X\rbrack_{t,t+h} + \lbrack Y\rbrack_{t,t+h}$ and use similar steps to deduce that $\lbrack X+Y\rbrack_{t,t+h} = \mathcal{O}_p(h)$.
\end{proof}

$\mathcal{S}^\text{CBP}_t$ is a proper subspace of the space of continuous semimartingales, see Example \ref{ex: not all martingales have CBP} below. Nevertheless, many processes of interest do lie in it. In particular, all It\^o processes with reasonably well-behaved coefficients have chords bounded in probability. This is Example \ref{ex: ito processes have CBP}.

\begin{example}[Not all martingales have chords bounded in probability] \label{ex: not all martingales have CBP}
	For a Brownian motion $\{W_t\}_{t\geq 0}$, let $M_t = W_{\sqrt{t}}$, for each $t\geq 0$. Then $M$ does not have chords bounded in probability at $t=0$. Indeed, $\lbrack M\rbrack_h = \sqrt{h} \neq \mathcal{O}(h)$ and so by Lemma \ref{lemma: CBP is equivalent to vbp for local martingales}, we see that $M\notin \mathcal{S}^\text{CBP}_t$.
\end{example}

\begin{example}[Reasonably well-behaved It\^o processes have chords bounded in probability] \label{ex: ito processes have CBP}
	Let $X$ be an It\^o process of the form $X = X_0 + \int_{0}^{\cdot}\mu_u du + \int_{0}^{\cdot}\sigma_u dW_u$. Suppose that $\mu$ is continuous with probability one or $L^1$-continuous at $t$. Similarly, suppose that $\sigma$ is continuous with probability one or $L^2$-continuous at $t$. Then $X\in\mathcal{S}^\text{CBP}_t$.
\end{example}

\begin{proof}
	See Appendix \ref{appendix: proof of results}.
\end{proof}

\subsection{The zero differential}

\begin{definition}[Zero differential] \label{def: zero differential d^p}
Let $X$ be a continuous semimartingale with decomposition $X=X_0 + A +M$. Furthermore, let $A=A^1-A^2$ be the canonical decomposition of the finite variation part as a difference of non-decreasing processes. We say that $X$ has \emph{differential zero} at time $t$ if
\begin{equation*}
	A^1_{t,t+h} = o_p(h),~ A^2_{t,t+h} = o_p(h) \text{ and }(M^*_{t,t+h})^2 = o_p(h).
\end{equation*}
Whenever $X$ has differential zero, we write $d^p(X)_t = 0$. 
\end{definition}

\begin{lemma} \label{lemma: A has differential 0 iff total variation is little op h}
	Let $A$ be a continuous finite variation process with $A_0 = 0$. Then 
	\begin{equation*}
	    d^p(A)_t = 0 \iff TV(A)_{t,t+h} = o_p(h).
	\end{equation*}
\end{lemma}

\begin{proof}
	The unique splitting of $A$ as the difference of two continuous non-decreasing processes beginning at zero is 
	\begin{equation*}
		A_t = \underbrace{\frac{1}{2}\left(TV(A)_t+A_t\right)}_{A^1_t} - \underbrace{\frac{1}{2}\left(TV(A)_t-A_t\right)}_{A^2_t}.
	\end{equation*} 
	Suppose $d^p(A)_t = 0$, i.e. $A^1_{t,t+h} = o_p(h)$ and $A^2_{t,t+h} = o_p(h)$. Then $TV(A)_{t,t+h} = A^1_{t,t+h} + A^2_{t,t+h} = o_p(h)$. Conversely, if $TV(A)_{t,t+h}=o_p(h)$ then $A_{t,t+h} = o_p(h)$ since $\lvert A_{t,t+h}\rvert \leq TV(A)_{t,t+h}$. Hence $\frac{1}{2}\left(TV(A)_{t,t+h}\pm A_{t,t+h}\right) = o_p(h) \Rightarrow d^p(A)_t = 0$.
\end{proof}

\begin{remark}
If a differentiable deterministic function $f$ has $d^p(f)_t = 0$ then $f'(t)=0$. This is because 
\begin{equation*}
	\lim\limits_{h\rightarrow 0^+}\frac{\lvert f(t+h)-f(t)\rvert}{h} \leq \lim\limits_{h\rightarrow 0^+}\frac{TV(f)_{t,t+h}}{h} = 0 \Rightarrow f'(t)=0.
\end{equation*}
The converse however need not to be true. See for instance Example \ref{ex: pathological example} in Appendix \ref{appendix: useful examples}. This means that our differential zero is not a perfect generalisation of the differential zero of the deterministic case. One might instead propose the alternative definition for the differential zero where we only need the finite variation part of the process to satisfy $A_{t,t+h} = o_p(h)$. This would generalise the differential zero from the deterministic case. However, we could find examples of processes that have a differential zero but do not have chords bounded in probability. This suggests that the stochastic versions of either the Fundamental Theorem of Calculus or the chain rule would break down in with this alternative definition.
\end{remark}

\begin{lemma} \label{lemma: M has differential zero iff its QV has differential zero}
	Let $M$ be a continuous local martingale. Then
	\begin{equation*}
		(M^*_{t,t+h})^2 = o_p(h) \iff \lbrack M\rbrack_{t,t+h} = o_p(h).
	\end{equation*}
	That is, $$d^p(M)_t = 0 \iff d^p\left(\lbrack M\rbrack\right)_t = 0.$$
\end{lemma}

\begin{proof}
	We start with the `$\Rightarrow$' part. Pick some $\epsilon, \tilde{\delta}>0$ and let $h>0$. Set $\beta^2\lambda^2 = \epsilon h$ and $\delta^2 \lambda^2 = \tilde{\epsilon} h$ in Corollary \ref{cor: good lambda inequality}. Then $\delta^2 = \frac{\beta^2 \tilde{\epsilon}}{\epsilon}$ and 
	\begin{equation*}
		\mathds{P}\left(\lbrack M\rbrack_{t,t+h}>\epsilon h,\left(M^*_{t,t+h}\right)^2<\tilde{\epsilon}h\right) \leq 4\frac{\tilde{\epsilon}}{\epsilon}\left(1+\frac{1}{\beta^2-1}\right)\mathds{P}\left(\lbrack M\rbrack_{t,t+h} \geq \frac{\epsilon h}{\beta^2}\right).
	\end{equation*}
	We can pick $\tilde{\epsilon}$ and $\beta$ such that $4\frac{\tilde{\epsilon}}{\epsilon}\left(1+\frac{1}{\beta^2-1}\right) <\frac{\tilde{\delta}}{2}$. For such $\tilde{\epsilon}$, there exists $\eta>0$ such that
	\begin{equation*}
		\mathds{P}\left(\left(M^*_{t,t+h}\right)^2>\tilde{\epsilon}h\right) < \frac{\tilde{\delta}}{2},~\forall h\in(0,\eta).
	\end{equation*}
	Hence,
	\begin{equation*}
		\mathds{P}\left(\lbrack M\rbrack_{t,t+h} > \epsilon h\right) < \tilde{\delta},~\forall h\in(0,\eta).
	\end{equation*}
	For the `$\Leftarrow$' part, we proceed in a similar fashion. Applying the Corollary \ref{cor: good lambda inequality} again, we get
	\begin{equation*}
		\mathds{P}\left(\left(M^*_{t,t+h}\right)^2>\epsilon h, \lbrack M\rbrack_{t,t+h} < \tilde{\epsilon}h \right) \leq \frac{\tilde{\epsilon}}{\epsilon}\left(1+\frac{1}{(\beta-1)}\right)^2\mathds{P}\left(\left(M^*_{t,t+h}\right)^2 \geq \frac{\epsilon h}{\beta^2}\right).
	\end{equation*}
	We can pick $\tilde{\epsilon}$ and $\beta$ such that $\frac{\tilde{\epsilon}}{\epsilon}\left(1+\frac{1}{(\beta-1)}\right)^2<\frac{\tilde{\delta}}{2}$. For such $\tilde{\epsilon}$, there exists $\eta>0$ such that
	\begin{equation*}
		\mathds{P}\left(\lbrack M\rbrack_{t,t+h} > \tilde{\epsilon}h\right) < \frac{\tilde{\delta}}{2},~\forall h\in(0,\eta).
	\end{equation*}
	Hence,
	\begin{equation*}
		\mathds{P}\left(\left(M^*_{t,t+h}\right)^2>\epsilon h\right)<\tilde{\delta},~\forall h\in(0,\eta).
	\end{equation*}
\end{proof}

\begin{lemma}[Alternative characterisation of the differential zero] \label{lemma: alternative characterisation of the differential zero}
	\begin{equation*}
		d^p(X)_t \iff TV(A)_{t,t+h} = o_p(h) \text{ and } \lbrack M\rbrack_{t,t+h} = o_p(h).
	\end{equation*}
\end{lemma}

\begin{proof}
This is a direct application of Lemmas \ref{lemma: A has differential 0 iff total variation is little op h} and \ref{lemma: M has differential zero iff its QV has differential zero}.
\end{proof}

\begin{corollary}
	\begin{equation*}
	    d^p(X)_t = 0 \Rightarrow X \in\mathcal{S}^\text{CBP}_t.
	\end{equation*}
\end{corollary}

\begin{proof}
This follows from Lemma \ref{lemma: VBP is equivalent to CBP}, Lemma \ref{lemma: alternative characterisation of the differential zero} above and the fact that being $o_p(h)$ implies being $\mathcal{O}_p(h)$.
\end{proof}

\begin{remark}
Our choice of definitions for $\mathcal{S}^\text{CBP}_t$ and the zero differential aim to define the stochastic differential for as large a subclass of continuous semimartingales while keeping certain properties true. One of these is that Lemma \ref{lemma: uniqueness of solutions} should hold. The other is that stochastic versions of the Fundamental Theorem of Calculus and of the chain rule should hold too. In particular, if we take a finite variation process $A$ which is `differentiable' according to our definitions, then $f(A)$ should also be differentiable for any continuously differentiable function $f:\mathbb{R}\rightarrow \mathbb{R}$. Comparing Definitions \ref{def: chords bounded in probability} and \ref{def: zero differential d^p}, one might be tempted to say $A$ is `differentiable' if $A^*_{t,t+h} = \mathcal{O}_p(h)$. This indeed would generalise Definition \ref{def: chords bounded in probability} but would fail to be a closed space under composition with $\mathcal{C}^1$ functions. We give an explicit example in Appendix \ref{appendix: useful examples} of a deterministic function $f$ with $f^*_{t,t+h}=\mathcal{O}_p(h)$ which is not in $\mathcal{S}^\text{CBP}_t$. We also show that it fails to keep this property under composition with $\mathcal{C}^1$ functions in general.
\end{remark}

\subsection{The differential as an equivalence relation}

We say that two continuous semimartingales with CBP at time $t$ have the same differential at $t$ if the process given by their difference has differential zero at $t$. This defines a binary relation on the space $\mathcal{S}^\text{CBP}_t$. We show this is an equivalence relation which allows us to formally define differentials as the equivalence classes of $\mathcal{S}^\text{CBP}_t$.

\begin{definition} \label{def: equivalence relation for d^p}
	Let $X$ and $Y\in\mathcal{S}^\text{CBP}_t$. We write that $d^p(X)_t = d^p(Y)_t \iff d^p(X-Y)_t = 0$.
\end{definition}

\begin{remark}
	In ordinary calculus this implies that $|(f-g)'(t)|=0$ from which we deduce that $f'(t)=g'(t)$.
\end{remark}

\begin{lemma} \label{lemma: tilde t is an equivalence relation}
	The binary relation `$d^p(X)_t = d^p(Y)_t$' defines an equivalence relation on $\mathcal{S}^\text{CBP}_t$.
\end{lemma}
\begin{proof}
	Symmetry and reflexivity are clear. Let $X$, $Y$ and $Z\in\mathcal{S}^\text{CBP}_t$. Suppose $d^p(X)_t = d^p(Y)_t$ and $d^p(Y)_t = d^p(Z)_t$. This means $d^p(X-Y)\sim_t 0$ and $d^p(Y-Z) \sim_t 0$. Suppose the processes $(X-Y)$ and $(Y-Z)$ have canonical decompositions given by $(X-Y) = (X-Y)_0 + A + M$ and $(Y-Z) = (Y-Z)_0 + B + N$ respectively. By definition of the differential zero, we have that $TV(A)_{t,t+h}, TV(B)_{t,t+h} = o_p(h)$ and $(M^*_{t,t+h})^2, (N^*_{t,t+h})^2 = o_p(h)$. Using Lemma \ref{lemma: properties of little and big o nonation in probability} and that 
	\begin{align*}
		TV(A+B)_{t,t+h} &\leq TV(A)_{t,t+h}+TV(B)_{t,t+h}, \\
		\left((M+N)^*_{t,t+h}\right)^2 &\leq 2\left(M^*_{t,t+h}\right)^2+2\left(N^*_{t,t+h}\right)^2,
	\end{align*}
	we deduce that $d^p(X-Z)_t= 0$. Thus $d^p(X)_t = d^p(Z)_t$. That is, the binary relation is transitive.
\end{proof}

\subsection{The space of It\^o stochastic differentials I}

We quotient out the space of continuous semimartingales with chords bounded in probability  by the equivalence relation of Definition \ref{def: equivalence relation for d^p}, call it $\sim_t$ for short, to define our notion of stochastic differentials.
\begin{definition}[Space of stochastic differentials]
	Let the \emph{space of stochastic differentials of $\mathbb{R}^m$-valued processes at }$t$ be 
	\begin{equation*}
		\mathds{D}_t\left(\mathbb{R}^m\right) \coloneqq  \sfrac{\mathcal{S}^\text{CBP}_t\left(\mathbb{R}^m\right)}{\sim_t}.
	\end{equation*}
	For a process $X\in\mathcal{S}^\text{CBP}_t\left(\mathbb{R}^m\right)$, we denote its corresponding element in $\mathds{D}_t\left(\mathbb{R}^m\right)$ by $d^p(X)_t$ just as we have done so far. Note that in Section \ref{section: financial applications}, we also denote elements of $\mathds{D}_t\left(\mathbb{R}^m\right)$ by greek letters like $\eta$. When $m=1$, we simply write $\mathds{D}_t$ for $\mathds{D}_t(\mathbb{R})$.
\end{definition}
Differentials of processes at time $t$ are equivalence classes of the space of processes with chords bounded in probability at time $t$. We now define operations on $\mathds{D}_t$ consistent with the ones that the usual differential $dX_t$ obeys and similar to those considered by It\^o in \cite{ito1975stochasticdifferentials}. However note that It\^o considered differentials as random valued functions of time intervals and not of time instants as we do. This difference is important because it makes our notion of the differential truly infinitesimal at an instant in time.

We write $H\in\mathcal{F}_t\left(\mathbb{R}\right)$ to mean that $H$ is a $\mathbb{R}$-valued, $\mathcal{F}_t$-measurable random variable.

\begin{definition} \label{def: operations on differentials} Let $d^p(X)_t$, $d^p(Y)_t \in\mathds{D}_t\left(\mathbb{R}^m\right)$ and $H\in\mathcal{F}_t(\mathbb{R})$. We define the operations of \emph{addition} (A), \emph{product} (P) and \emph{multiplication by} $\mathcal{F}_t$-\emph{measurable} (M) as
	\begin{align*}
		d^p(X)_t \oplus d^p(Y)_t &\coloneqq d^p(X+Y)_t, \tag{A}\\
		d^p(X)_t \star d^p(Y)_t &\coloneqq d^p\left(\lbrack X,Y\rbrack\right)_t, \tag{P}\\
		H d^p(X)_t &\coloneqq d^p(H X)_t, \tag{M}
	\end{align*}
\end{definition}
In actual fact, we can think of (A) just as the usual sum of sets.
\begin{lemma} \label{lemma: linearity of differentials}
	For any $d^p(X)_t$, $d^p(Y)_t\in\mathds{D}_t\left(\mathbb{R}^m\right)$, $d^p(X)_t+ d^p(Y)_t=d^p(X+Y)_t=d^p(X)_t\oplus d^p(Y)_t$.
\end{lemma}
For the remainder of this subsection, all capital letters refer to processes and not random variables. We drop the time index for brevity.
\begin{proof}
  Let $X\in d^p(X)_t$ and $Y\in d^p(Y)_t$. By definition, this means $d^p(X-\tilde{X})_t= 0$ and $d^p(Y-\tilde{Y})_t= 0$. By transitivity, we deduce that 
	\begin{equation*}
		d^p\left((X+Y)-(\tilde{X}+\tilde{Y})\right)_t = 0.
	\end{equation*}
	That is $X+Y\in d^p(X+Y)_t \Rightarrow d^p(X)+d^p(Y)_t \subseteq d^p(X+Y)_t$. For the reverse inclusion, we prove that $d^p(X)_t+d^p(Y)_t$ is an equivalence class itself. This suffices because if an equivalence class is contained in another, they must be one and the same. Let $\tilde{Z}$ and $\hat{Z}$ be arbitrary elements in $d^p(X)_t +d^p(Y)_t$. Then there exist $\tilde{X}$, $\hat{X}\in d^p(X)_t$, and $\tilde{Y}$, $\hat{Y}\in d^p(Y)_t$ such that $\tilde{Z} = \tilde{X}+\tilde{Y}$ and $\hat{Z}=\hat{X}+\hat{Y}$. Then $d^p(\tilde{Z}-\hat{Z})_t=d^p\left((\tilde{X}-\hat{X})+(\tilde{Y}-\hat{Y})\right)_t = 0$ since $d^p(\tilde{X}-\hat{X})_t = 0$ and $d^p(\tilde{Y}-\hat{Y})_t = 0$. Hence $d^p(\tilde{Z})_t = d^p(\hat{Z})_t = 0$.
\end{proof}

\begin{lemma} \label{lemma: the zero differential is absorbing under operation star}
	\begin{equation*}
		\eta \star 0 = 0 = 0 \star \eta,~\forall \eta \in\mathds{D}_t\left(\mathbb{R}^m\right),
	\end{equation*}
	where $0$ is understood to be the zero differential in $\mathds{D}_t\left(\mathbb{R}^m\right)$.
\end{lemma}

\begin{proof}
	Suppose $\eta \in \mathds{D}_t\left(\mathbb{R}^m\right)$. Then there exists $Y\in \mathcal{S}^\text{CBP}_t\left(\mathbb{R}^m\right)$ such that $\eta = d^p(Y)_t$. There is also some $X\in\mathcal{S}^\text{CBP}_t\left(\mathbb{R}^m\right)$ such that $d^p(X)_t = 0$. We want to show that $d^p(X)_t \star d^p(Y)_t = 0$ which by definition means that we want to show $d^p\lbrack X,Y\rbrack_t = 0$. It is known that for any $0\leq t\leq u$, we have
	\begin{equation*}
		\left \lvert \lbrack X,Y\rbrack_{t,u} \right \rvert \leq \left(\lbrack X,X\rbrack_{t,u}\lbrack Y,Y\rbrack_{t,u}\right)^{\frac{1}{2}}.
	\end{equation*}
	Hence, $\left \lvert \lbrack X,Y\rbrack_{t,u} \right \rvert \leq \left(o_p(\lvert u-t\rvert )\mathcal{O}_p(\lvert u-t\rvert )\right)^{\frac{1}{2}} = o_p(\lvert u-t\rvert)$ by Lemma \ref{lemma: properties of little and big o nonation in probability}. Thus,
	$\lbrack X,Y\rbrack_{t,t+h} \leq o_p(h)$.
\end{proof}

\subsection{It\^o stochastic differential equations}

We now show that if a semimartingale has a differential of zero over the time interval $[0,T]$ then it is constant in time. Note that it does not imply that it is a constant, but rather that it is equal to its initial value at time $0$, which might be a random variable.

\begin{lemma}[Uniqueness of solutions] \label{lemma: uniqueness of solutions}
	If $d^p(X)_t = 0 $ for all $t\in[0,T]$, then $X_t=X_0$, for all $t\in[0,T]$. 
\end{lemma}

\begin{proof}
	Let $X$ have canonical decomposition $X_t = X_0 + A_t+M_t$ as before. Then $TV(A)_{t,t+h}=o_p(h)$ for all $t\in[0,T]$. In particular, this implies that the sequence $\left\{n\cdot TV(A)_{t,t+\frac{1}{n}}\right\}_{n\in\mathbb{N}}$ converges to zero in probability. Convergence in probability implies almost sure convergence on a subsequence. Hence, for each $t\in[0,T]$, there exists a subsequence $\left\{n_k(t)\right\}_{k\geq 1}$ such that
	\begin{equation*}
		n_k(t)\cdot TV(A)_{t,t+\frac{1}{n_k(t)}}(\omega)\rightarrow 0 \text{, as }k\rightarrow\infty \text{ and for a.a. }\omega\in\Omega.
	\end{equation*}
	That is, for almost all $\omega\in\Omega$ and for each $\epsilon>0$, there exists a large enough $K=K(\omega,\epsilon,t)\in\mathds{N}$ such that
	\begin{equation}
		n_k(t)\cdot TV(A)_{t,t+\frac{1}{n_k(t)}}(\omega)<\frac{\epsilon}{3T},~\forall k\geq K(\omega,\epsilon,t) \label{eq: almost sure bound on chords of A uniqueness lemma}
	\end{equation}
	For arbitrarily small $\delta>0$, we have
	\begin{equation*}
		\underset{t\in[0,T]}{\bigcup}\left(t,t+n^{-1}_{K(\omega,\epsilon,t)}(t)\right)\supseteq [\delta,T].
	\end{equation*}
	By the Heine--Borel theorem, a finite number of such open intervals cover $[\delta,T]$, call them $\left\{(t_i,s_i)\right\}_{i=0}^m$ where $t_0<\delta$ and $s_m>T$. Without loss of generality, we can assume that none of these open interval is a subinterval of any other (or else we simply remove them). Then, $\sum_{i=1}^{m}(s_i-t_i)< 3T$. Assume also that we chose $\delta\in(0,n^{-1}_{K(\omega,\epsilon,0)}(0))$ so that $\bigcup_{i=1}^m\left(t_i,s_i\right) \supseteq (0,T]$. Applying (\ref{eq: almost sure bound on chords of A uniqueness lemma}), we obtain that
	\begin{equation*}
		TV(A)_{0,T}(\omega) \leq \sum_{i=1}^{m}TV(A)_{t_i,s_i}(\omega) < \frac{\epsilon}{3T}\sum_{i=1}^{m}(s_i-t_i) < \epsilon.
	\end{equation*}
	The choice of $\epsilon>0$ was arbitrary so $A^*_{0,T}(\omega)=0$ which implies that $A_t(\omega) = A_0(\omega)$ for all $t\in[0,T]$. This is true for almost all $\omega\in\Omega$ so $A_t=0$ almost surely, for all $t\in[0,T]$. Using Lemma \ref{lemma: M has differential zero iff its QV has differential zero} and following similar steps we can show that $\lbrack M\rbrack_{0,T} = \lbrack X\rbrack_{0,T}=0$ almost surely. By Lemma $5.14$ of \cite{legall2016martingales}, this implies that $M_t=0$ for all $t\in[0,T]$ which completes the proof.
\end{proof}

We prove a stochastic version of Fundamental Theorem of Calculus but first we need a preliminary lemma to compute the total variation of integrals with respect to processes of finite variation. This result is most likely not new but we prove it here for completeness.

\begin{lemma}
	Let $H$ be a continuous and adapted process and $A$ be a continuous process of finite variation beginning at zero. Then
	\begin{equation*}
		TV\left(\int_{0}^{\cdot}H_udA_u\right) = \int_{0}^{t}\lvert H_u\rvert dTV(A)_u.
	\end{equation*}
\end{lemma}

\begin{proof}
	We use the notation $x^+ \coloneqq \max(0,x)$ and $x^- \coloneqq -\min(0,x)$. $A$ can uniquely be split as a difference of non-decreasing processes. Explicitly the splitting is given by
	\begin{equation*}
		A_t = \underbrace{\frac{1}{2}\left(TV(A)_t + A_t\right)}_{=A^1_t} - \underbrace{\frac{1}{2}\left(TV(A)_t - A_t\right)}_{=A^2_t}.
	\end{equation*}
Consequently its total variation is given by the sum of these processes. All we need to do is split the integral $\int_{0}^{t}H_udA_u$ as a difference of non-decreasing processes and sum those to get the result. We have
\begin{align*}
	\int_{0}^{t}H_udA_u &= \int_{0}^{t}\left(H^+_u-H^-u\right)d\left(A^1-A^2\right)_u \\ 
	&= \int_{0}^{t}\left(H^+_udA^1_u + H^-_udA^2_u\right) -  \int_{0}^{t}\left(H^-dA^1_u + H^+dA^2_u\right), \intertext{and so}
	TV\left(\int_{0}^{\cdot}H_udA_u\right)_t &= \int_{0}^{t}\left(H^+_udA^1_u + H^-_udA^2_u\right)  +  \int_{0}^{t}\left(H^-dA^1_u + H^+dA^2_u\right) \\
	&= \int_{0}^{t}\lvert H_u\rvert d\left(A^1+A^2\right)_u = \int_{0}^{t}\lvert H_u\rvert dTV(A)_u,
\end{align*}
where we use that $x^+ + x^- = \lvert x\rvert$.
\end{proof}

\begin{theorem}[Stochastic Fundamental Theorem of Calculus] \label{thm: stochastic version of ftc}
	Let $X\in\mathcal{S}^\text{CBP}_t$ and $H$ be an adapted, sample continuous process. Then the stochastic integral $\int_{0}^{\cdot}H_udX_u \in\mathcal{S}^\text{CBP}_t$ and its differential satisfies
	\begin{equation*}
		d^p\left(\int_{0}^{\cdot}H_udX_u\right) = H_t d^p(X)_t.
	\end{equation*}
\end{theorem}

\begin{proof}
	Let $X=X_0+A+M$ have its usual decomposition. The fact that $\int_{0}^{\cdot}H_udX_u$ is a continuous semimartingale over $[0,T]$ follows from the definition of the It\^o integral. The process $\left\{H_tX_{t,s}:s\geq t\right\}\in\mathcal{S}^\text{CBP}_t$ since $X\in\mathcal{S}^\text{CBP}_t$. Thus, it is necessary and sufficient to prove that the process $\left\{Y_s \coloneqq \int_{t}^{s}H_udX_u -  H_tX_{t,s}\right\}_{s\geq t}$  satisfies $d^p(Y)_t = 0$. This process has canonical decomposition $Y_s = Y_t+Y_s^1+Y_s^2$ where $Y_s^1 = \int_{t}^{s}H_{t,u}dA_u$ and $Y_s^2 = \int_{t}^{s}H_{t,u}dM_u$ are its finite variation and martingale parts respectively. We have that
	\begin{equation*}
		 TV\left(Y^1\right)_{t,t+h}  =  \int_{t}^{t+h}\lvert H_{t,u}\rvert d\left(TV(A)\right)_u \leq H^*_{t,t+h}\cdot TV(A)_{t,t+h} = o_p(h),
	\end{equation*} 
	using that $TV(A)_{t,t+h} = \mathcal{O}_p(h)$ from Lemma \ref{lemma: VBP is equivalent to CBP}, $H^*_{t,t+h} = o_p(1_h)$ from Lemma \ref{lemma: sup of continuous process is op1} and finally Lemma \ref{lemma: properties of little and big o nonation in probability}. For the martingale part, we use that $\lbrack Y^2\rbrack_{t,t+h}=\int_{t}^{t+h}(H_{t,u})^2d\lbrack M\rbrack_u$ and similar steps,to find that $\lbrack Y^2\rbrack_{t,t+h} = o_p(h)$. Hence $d^p(Y)_t = 0$.
\end{proof}

We require that the integrator has variations bounded in probability and this condition may not be weakened. Here is an example where the theorem fails if this condition is not met, even for a differentiable deterministic function.

\begin{example}
	Consider $H_t = \sqrt{t}$ and $X_t = \sqrt{t}$. Then 
	\begin{equation*}
		\int_{0}^{t}H_u dX_u = \int_{0}^{t}\sqrt{u}d\left(\sqrt{u}\right) = \frac{1}{2}\int_{0}^{t}du = \frac{t}{2}. 
	\end{equation*}
In this case,
\begin{equation*}
	d^p\left(\int_{0}^{t}H_u dX_u\right)_0 = \frac{1}{2}d^p(t)_0 \neq 0 
\end{equation*}
even though $H_0 = 0$. Thus the Fundamental Theorem of Calculus does not hold in this case. While we can still define $d^p(X)_t$ even if $X$ is not differentiable, it will not obey the FTC. The differential $d^p\left(\int_{0}^{\cdot}H_udX_u\right)_0$ depends upon information of $H$ in a neighbourhood of $t=0$, not just at $t=0$ if $H$ is not differentiable.
\end{example}

\begin{theorem}[Chain rule] \label{thm: stochastic chain rule}
	Let $f\in\mathcal{C}^2$ and $X\in\mathcal{S}^\text{CBP}_t$. Then $f(X) \in \mathcal{S}^\text{CBP}_t$ and 
	\begin{equation*}
		d^p\left(f(X)\right)_t = f'(X_t)d^p(X)_t + \frac{1}{2}f''(X_t)d^p\left(\lbrack X\rbrack\right)_t.
	\end{equation*}
\end{theorem}

\begin{proof}
	Define the process $Y_s \coloneqq f(X_s) - f'(X_t)X_s - \frac{1}{2}f''(X_t)\lbrack X\rbrack_s$ for $s\geq t$. As $X\in\mathcal{S}^\text{CBP}_t$, we know that $\lbrack X\rbrack \in\mathcal{S}^\text{CBP}_t$ by Lemma \ref{lemma: CBP is equivalent to vbp for local martingales}. Thus, it is sufficient to show that $d^p(Y)_t = 0$. Pick some $s>t$. By the classical It\^o formula, 
	\begin{multline*}
		Y_s = Y_t + \underbrace{\left\{\int_{t}^{s}\left(f'(X_u)-f'(X_t)\right)dA_u + \frac{1}{2}\int_{t}^{s}\left(f''(X_u)-f''(X_t)\right)d\lbrack X\rbrack_u\right\}}_{\text{finite variation part}} \\
		+ \underbrace{\int_{t}^{s}\left(f'(X_u)-f'(X_t)\right)dM_u}_{\text{local martingale part}}. 
	\end{multline*}
	By continuity of $f'$ and $f''$ and using the same steps as in the proof of Theorem \ref{thm: stochastic version of ftc}, we have the result.
\end{proof}

\begin{remark}
	This proof uses the classical It\^o formula as we require the canonical splitting of $f(X)$ to even be able to apply our definition of the differential zero. This is given by the classical It\^o formula. Our result is nevertheless new as it provides a `true' chain rule which applies at a point in time and not over an interval. 
\end{remark}



\begin{corollary}[Leibniz/product rule for differentials]
    Let $M$ and $N$ be continuous local martingales with $M$, $N\in\mathcal{S}^{\text{CBP}}_t$. Then $MN\in\mathcal{S}^{\text{CBP}}_t$ and
    \begin{equation*}
        d^p\left(MN\right)_t = M_td^p(N)_t + N_td^p(M)_t + d^p\left(\lbrack M,N\rbrack \right)_t.
    \end{equation*}
\end{corollary}
\begin{proof}
    By linearity, $M+N\in\mathcal{S}^{\text{CBP}}_t$. By Lemma \ref{lemma: linearity of differentials}, the operation (M) of Definition \ref{def: operations on differentials} and the stochastic chain rule with $f(x)=x^2$, we have
    \begin{align*}
        &d^p(MN)_t \\
        &= \frac{1}{2}\left(d^p\left((M+N)^2\right)_t-d^p(M^2)_t-d^p(N^2)\right) \\
        &= d^p\left((M_t+N_t)(M+N)-M_tM-N_tN\right)_t + \frac{1}{2}d^p\left(\lbrack M+N\rbrack -\lbrack M\rbrack - \lbrack N\rbrack \right)_t \\
        &= M_td^p(N)_t + N_td^p(M)_t + d^p\left(\lbrack M,N\rbrack \right)_t,
    \end{align*}
    where we have used that $\lbrack M,N\rbrack \coloneqq \frac{1}{2}\left(\lbrack M+N\rbrack - \lbrack M\rbrack -\lbrack N\rbrack\right)$ in the last line.
\end{proof}
\begin{corollary} \label{cor: corollary to the ito product rule for differentials}
	For any $d^p(X)_t$ and $d^p(Y)_t\in\mathds{D}_t$, 
	\begin{equation*}
		d^p(X)_t\star d^p(Y)_t = d^p(XY)_t-X_t d^p(Y)_t-Y_t  d^p(X)_t.
	\end{equation*}
\end{corollary}
\begin{proof}
    This follows directly from the definition of operation (P) in Definition \ref{def: operations on differentials}.
\end{proof}

We now show that we may use $d^p(X)_t$ instead of $dX_t$ for SDEs with continuous coefficients.
\begin{theorem}\label{theorem: dX_t and dpX_t are equivalent}
	The following are equivalent for any continuous semimartingale in $\mathcal{S}^\text{CBP}_{[0,T]}\big(\mathbb{R}\big)$:
	\begin{enumerate}[label=(\roman*), itemsep=0pt, topsep=0pt]
		\item The process $\{Y_t\}_{t\in[0,T]}$ satisfies $Y_0 = y$ and the It\^o stochastic integral equation
		\begin{equation*} 
			dY_t = f(t,X_t) dX_t,~ \forall t\in[0,T].
		\end{equation*}
		\item The process $\{Y_t\}_{t\in[0,T]}$ satisfies $Y_0 = y$ and the It\^o stochastic differential equation
		\begin{equation*} 
			d^p(Y)_t = f(t,X_t) d^p(X)_t,~\forall t\in[0,T].
		\end{equation*}
	\end{enumerate}
\end{theorem}

\begin{proof}
	We show $(ii) \Rightarrow (i)$ first. By transitivity of $d^p(\cdot)_t$ and the Fundamental Theorem of Stochastic Calculus, get that
	\begin{equation*}
		d^p\big(Y - \int_{0}^{\cdot}f(s,X_s) dX_s\big)_t = 0,~\forall t\in[0,T].
	\end{equation*}
	By Lemma \ref{lemma: uniqueness of solutions}, we deduce that
	\begin{equation*}
		Y_t - \int_{0}^{t}f(s,X_s) dX_s = Y_0 = y,~\forall t\in[0,T].
	\end{equation*}
	Hence $Y$ satisfies $(i)$. The reverse implication is similar. If $Y$ satisfies the above, then it is clear we have
	\begin{equation*}
		d^p\Big(Y - \int_{0}^{\cdot}f(s,X_s) dX_s\Big)_t = 0,~\forall t\in[0,T].
	\end{equation*} 
	By transitivity and the Fundamental Theorem of Stochastic Calculus, we find that
	\begin{equation*}
		d^p(Y)_t = d^p\Big(\int_{0}^{\cdot}f(s,X_s) dX_s\Big)_t = f(t,X_t) d^p(X)_t,~\forall t\in[0,T],
	\end{equation*}
	which completes the proof.
\end{proof}

\begin{corollary}[Existence and uniqueness of solutions of It\^o stochastic differential equations]
	Let $\{W_t\}$ be a Brownian motion on a filtered probability space $(\Omega,\mathcal{A},\{\mathcal{F}_t\},\mathds{P})$. Suppose $\mu(t,X_t),\sigma(t,X_t)$ are measurable functions which satisfy linear growth and Lipschitz conditions and let $x$ be a random variable with finite second moment. Then, the It\^o stochastic differential equation
	\begin{equation*}\label{eq: true sde uniqueness of solutions}
	\left\{
                \begin{array}{rcl}
                  d^p(X)_t &= &\mu(t,X_t) d^p\big(\lbrack W,W\rbrack\big)_t + \sigma(t,X_t) d^p(W)_t,~\forall t\in[0,T],\\
                  X_0 &= &x,
                \end{array}
              \right.
    \end{equation*}
	has as a $\mathds{P}$-almost surely unique sample continuous 
	solution. 
\end{corollary}
\begin{proof}
	This follows from the result above and the usual Theorem on existence and uniqueness of solutions of stochastic integral equations (see \cite{oksendal2013sdes}).
\end{proof}

\subsection{The space of It\^o stochastic differentials II}
We use the properties of It\^o stochastic differential equations proved in the previous section to deduce properties of the space of It\^o differentials. In particular, we show that the dimension of the space of differentials of martingales encodes the Martingale Representation Theorem.

\begin{definition}
	Define the following subspaces of $\mathds{D}_t$ :
	\begin{align*}
		\mathds{D}^A_t &\coloneqq  \Big\{d^p(A)_t\in\mathds{D}_t : \{A_s\}_{s\geq t} \text{ is a continuous finite variation process}\Big\}, \intertext{and}
		\mathds{D}^M_t &\coloneqq  \Big\{d^p(M)_t\in\mathds{D}_t : \{M_s\}_{s\geq t} \text{ is a continuous local martingale}\Big\}. 
	\end{align*}
\end{definition}

\begin{corollary}
	$(\mathds{D}_t,+,\star)$ is a commutative ring and $\mathds{D}^A_t$ a subring of $\mathds{D}_t$. The space $(\mathds{D}_t,+)$ is a commutative $\mathcal{F}_t$-module and $\mathds{D}^M_t$ an $\mathcal{F}_t$-submodule of $\mathds{D}_t$.
	Furthermore, the products of the rings satisfy
	\begin{equation*}
		\mathds{D}_t\star \mathds{D}_t \subset \mathds{D}^A_t,~\mathds{D}^A_t\star \mathds{D}_t = 0 \text{, and }\mathds{D}_t\star\mathds{D}_t\star\mathds{D}_t = 0.
	\end{equation*}
\end{corollary}

\begin{proof}
	The first part follows from Corollary \ref{cor: corollary to the ito product rule for differentials} and the axioms of rings and modules. The second part follows from that facts that the quadratic co-variation of continuous semimartingales is a continuous finite variation process and that the quadratic co-variation of finite variation processes vanishes.
\end{proof}

\begin{remark}
	Under certain circumstances, it is convenient to consider $\mathds{D}_t$ as a vector space although we must normally say it is an $\mathcal{F}_t$-module. Given the information up to time $t$, any $\mathcal{F}_t$-measurable random variable is just a real number. Therefore, $\mathds{D}_t\rvert_{\mathcal{F}_t}$ can then be interpreted as an ordinary real vector space.
\end{remark}

\begin{lemma}
	Let $\mathds{D}^\text{It\^o}_t$ denote the space of differentials at time $t$ of It\^o processes of the form of Example \ref{ex: ito processes have CBP} which are generated by a $d$-dimensional Brownian motion $W=(W^1,\dots ,W^d)$. Then the $\mathcal{F}_t$-module $\mathds{D}^\text{It\^o}_t$ is finitely generated and has dimension $d+1$.
\end{lemma}

\begin{proof}
	By the definition of It\^o processes, $\big\{dt,d^p(W^1)_t,\dots,d^p(W^d)_t\big\}$ is a generating set for $\mathds{D}^\text{It\^o}_t$. Note that $dt$ can be understood as $d^p\big([W^i,W^i]\big)_t$ for any of the $i\in\{1,\dots,d\}$. To show linear independence, suppose that there exists $\mathcal{F}_t$-measurable random variables $H^0,H^1,\dots,H^d$ such that 
	$H^0 dt+H^1 d^p(W^1)_t +\dots + H^d d^p(W^d)_t = 0$. Multiplying by $d^p(W^i)_t$ for some $i= 1,\dots, d$, we have $H^i  dt=0$ since $d^p(W^i)_t \star d^p(W^j)_t = \delta_{i,j} dt$ and $d^p(W^i)_t \star dt = 0$. Hence $H^i = 0$, $\forall i=1,\dots, d$. It follows also that $H^0=0$.
\end{proof}

\begin{lemma}
	Let $W=(W^1,\dots ,W^d)$ be a $d$-dimensional Brownian motion and let $\{\mathcal{G}_t\} = \{\mathcal{F}^W_t\}$ be its natural filtration. Let $\mathds{D}^{M;W}_{t}$ denote the space of differentials of continuous $\mathcal{G}_t$-local martingales at time $t$. Then the $\mathcal{G}_t$-module $\mathds{D}^{M;W}_t$ is finitely generated and has dimension $d$.
\end{lemma}
\begin{proof}
	By the Martingale Representation Theorem, $\big\{d^p(W^1)_t,\dots,d^p(W^d)_t\big\}$ is a generating set for $\mathds{D}^{M;W}_t$. To show linear independence, we use the same argument as the previous Lemma.
\end{proof}

\begin{corollary}
	The $\mathcal{G}_t$-module of differentials of continuous $\mathcal{G}_t$-semimartingales with absolutely continuous finite variation part has dimension $d+1$. 
\end{corollary}
\begin{proof}
	By the Martingale Representation Theorem and the assumption on the finite variation part, $\big\{dt,d^p(W^1)_t,\dots,d^p(W^d)_t\big\}$ is a generating set and linear independence is proved as above.
\end{proof}

At this point we can not say much more about the space of differentials of semimartingales whose finite variation part is not absolutely continuous. To see this consider the case of the deterministic and singular Cantor function $c:[0,1]\rightarrow [0,1]$. It satisfies the condition $c^*_{t,t+h}=o(h)$ at almost every $t\in[0,1]$ (because it has derivative zero almost everywhere and is increasing). But the uniqueness of solutions of Lemma \ref{lemma: uniqueness of solutions} and subsequently of Theorem \ref{theorem: dX_t and dpX_t are equivalent} would be violated. Therefore we only consider the case of absolutely continuous finite variation part in this paper.

We have shown that the space of differentials we consider are finite dimensional and spanned by the generating Brownian motions. If we consider strange enough integrands, we can show that for more general stochastic processes, the space of differentials should be infinitely dimensional. We illustrate this in the following example. Nevertheless, we argue that the processes constructed in this manner cannot be meaningfully hedged. As the results of this section are used in our financial applications of Section \ref{section: financial applications}, we do not consider such general processes.

\begin{example}[`Infinitely fast buying and selling at $0$']
	Suppose we work in one dimension and $W$ is our generating Brownian motion. Consider the (deterministic) hedging strategy $\phi(s) = (-1)^n\mathds{1}_{s\in[\nicefrac{1}{n+1},\nicefrac{1}{n})}$. Then $\phi$ is c\`adl\`ag so the stochastic integral $\int \phi(s)dW_s$ exists and is another Brownian motion, call it $\tilde{W}$. In this case, the Fundamental Theorem of Calculus at $t=0$ cannot hold; we do not have $d^p(W)_0 \propto d^p(\tilde{W})_0$. Hence $d^p(W)_0$ cannot span the space of differentials of general stochastic integrals, only a subspace of those. For our financial applications we need not consider all possible stochastic integrals as some are representations of unobtainable hedging strategies (it is not realistic to buy and sell infinitely quickly near time $0$ in this case); instead we think of processes whose stochastic differential exists as those which we can realistically hedge locally.
\end{example} 

We now define a module inner product on a subspace of $\mathds{D}_t^\text{It\^o}$ using an expectation. In order for the expectation to be well-defined, we need the differentials to have square integrable It\^o processes as representatives. We use this definition later to illustrate some financial applications of stochastic differentials in Section \ref{section: financial applications}. This is not technically an inner product as it takes values in the space of $\mathcal{F}_t$-measurable random variables and not in $\mathbb{R}$. Nevertheless, one can immediately consider it as an inner product if we condition on the information provided by $\mathcal{F}_t$. Alternatively we can see it is an inner product in the pathwise sense too.

\begin{definition} \label{def: conditional expectation of an ito differential}
	Let $\mathds{D}^\text{It\^o}_t(L^2)$ be the space of It\^o differentials of twice integrable processes. Let $\eta \in\mathds{D}^\text{It\^o}_t(L^2)$. Then there exists a twice integrable real-valued It\^o process $\{X_t\}$ such that $d^p(X)_t = \eta$. Define the conditional expected value of the differential as
	\begin{equation*}
		\mathbb{E}_t[\eta] \coloneqq \lim\limits_{h\rightarrow 0^+}\frac{\mathbb{E}\left[X_{t+h}-X_t|\mathcal{F}_t\right]}{h},
	\end{equation*}
	where the convergence is in the $L^1$-sense.
\end{definition}

\begin{definition}
	We call
	\begin{equation*}
		\ker(\mathbb{E}_t) \coloneqq \left\{\eta \in \mathds{D}^\text{It\^o}_t(L^2):\mathbb{E}_t[\eta]=0\right\}
	\end{equation*}
	the \textit{space of martingale differentials} on $\mathds{D}^\text{It\^o}_t(L^2)$.
\end{definition}
Then there exists a bilinear form $q^\perp:\mathds{D}^\text{It\^o}_t(L^2)\times \mathds{D}^\text{It\^o}_t(L^2) \rightarrow \mathcal{F}_t$ such that $q^\perp(dt,dt)=1$ and $q^\perp(\eta,\tilde{\eta})=0$ for any $\eta \in \ker(\mathbb{E}_t)$. In fact, it is not too hard to see that $q^\perp(\eta,\tilde{\eta})\coloneqq \mathbb{E}_t[\eta]\mathbb{E}_t[\tilde{\eta}]$ will do. We can now define a pathwise inner product on the space of It\^o differentials.

\begin{lemma} \label{lemma: inner product on space of ito differentials}
	Define the map $\langle \cdot,\cdot\rangle : \mathds{D}^\text{It\^o}_t(L^2)\times \mathds{D}^\text{It\^o}_t(L^2) \rightarrow \mathcal{F}_t(\mathbb{R})$ as
	\begin{equation*}
		\langle \eta,\tilde{\eta}\rangle \coloneqq \mathbb{E}_t[\eta]\mathbb{E}_t[\tilde{\eta}] + \underbrace{\mathbb{E}_t[\eta\star \tilde{\eta}]}_{\eqqcolon q(\eta,\tilde{\eta})}=q^\perp(\eta,\tilde{\eta}) + q(\eta,\tilde{\eta}).
	\end{equation*}
	Then $\langle \cdot,\cdot\rangle$ defines a pathwise inner product on $\mathds{D}^\text{It\^o}_t(L^2)$ in the sense that it is a symmetric bilinear form and it is almost-surely positive definite.
\end{lemma}

\section{Financial applications}
\label{section: financial applications}
In \cite{armstrong2018classifying}, the authors develop the notion of an isomorphism of markets. Thus it is natural to ask which financial concepts may be defined in an invariant manner, from the point of view of isomorphisms. In this section we use stochastic differentials to show that the notion of a portfolio of assets at a moment in time $t$ can be given an invariant meaning. We call these `instantaneous portfolios'. 

We define the instantaneous risk and return of such portfolios and see how Markowitz's classical theory may be understood as the theory of instantaneous portfolios. We see how the basic invariant of a market identified in \cite{armstrong2018classifying}, the absolute market price of risk, arises naturally by considering the Markowitz optimisation problem of instantaneous portfolios. Let us now recall the definitions of \cite{armstrong2018classifying}.

\begin{definition}[Multi-period market \cite{armstrong2018classifying}]
A multi-period market consists of the following:
\begin{enumerate}[label=(\roman*), itemsep=0pt, topsep=0pt]
	\item A filtered probability space $(\Omega,\mathcal{A},\mathcal{F}_t,\mathbb{P})$ where $t\in\mathcal{T}\subseteq [0,T]$ for some index set $\mathcal{T}$ containing both $0$ and $T$. We write $\mathcal{F}=\mathcal{F}_T$. We require that $\mathcal{F}_0=\{\emptyset,\Omega\}$.
	\item For each $X\in L^0(\Omega;\mathbb{R})$, an $\mathcal{F}_t$-adapted process $c_t(X)$ defined for each $t\in \mathcal{T}\setminus T$. 
\end{enumerate}
\end{definition}
Random variables $X\in L^0(\Omega,\mathcal{F}_t;\mathbb{R})$ are interpreted as contracts which have payoff $X$ at time $T$. The cost of such contracts at time $t$ are denoted $c_t(X)$.

\begin{definition}[Continuous time complete market \cite{armstrong2018classifying}]
	A continuous time market $(\Omega, \mathcal{A}, \mathcal{F}_t, \mathbb{P}, c_t)$ on $[0, T]$ is called a
	\emph{continuous time complete market} with risk free rate $r$ if there exists a measure
	$\mathbb{Q}$ equivalent to $\mathbb{P}$ with
	\begin{equation*}
	c_t(X) = e^{-r(T-t)} \mathbb{E}^\mathbb{Q}(X | \mathcal{F}_t)
	\end{equation*} 
	for $\mathbb{Q}$-integrable random variables $X$ and equal to $+\infty$ otherwise.
\end{definition} 

\begin{example}[Diffusion market \cite{armstrong2018classifying}] \label{ex: diffusion market}
	Let $(\Omega, \mathcal{A}, \mathcal{F}_t, \mathbb{P})$ be $d$-dimensional Wiener
	space, that is the probability space generated by the $d$-dimensional Brownian
	motion $\bm W_t$. A diffusion market is one in which all asset can be described by an $d$-dimensional stochastic process $\bm X_t$ defined by a
	stochastic differential equation of the form
	\begin{equation} \label{eq: diffusion market dynamics}
	d\bm X_t = \bm\mu(t,\bm X_t)dt + \bm\sigma(t,\bm X_t)d\bm W_t.
	\end{equation}
	Here $\bm\mu $ is a $\mathbb{R}^d$-valued function and $\bm\sigma$ is an invertible-matrix valued
	function. We assume the coefficients $\bm\mu$ and $\bm\sigma$ are sufficiently well-behaved for
	the solution of the equation to be well-defined on $[0, T ]$. The components, $X^i_t$,
	of the vector $\bm X_t$ are intended to model the prices of $d$ assets. 
	
	The diffusion market for (\ref{eq: diffusion market dynamics}) with risk-free rate $r$ over a time period $[0, T ]$ is
	given by defining $c_t : L_0(\Omega; \mathbb{R}) \rightarrow \mathbb{R}$ for $t \in [0, T)$ by
	\begin{equation*}
	 c_t(X) = 
	\begin{cases}
			\alpha_0e^{-r(T-t)} + \sum_{i=1}^{d}\alpha_iX^i_t, & \text{ if } X = \alpha_0 + \sum_{i=1}^{d}\alpha_i X_T^i,\\
            \infty, & \text{otherwise.}
		 \end{cases}
	\end{equation*}
\end{example}
This is well-defined so long as we assume that $X^i_T$ are linearly independent
random variables. This is the case in all situations of interest. The market defined above is called an exchange market because it models the basic assets that can be purchased directly on an exchange. 

Example \ref{ex: diffusion market} illustrates the classical coordinate approach to understanding markets. However, the decomposition of the market into basis assets corresponding to basis vectors of $\mathbb{R}^n$ is not invariant under market morphisms. 

Market morphisms are defined in \cite{armstrong2018classifying}. They can be understood financially as an embedding of one market into another. Isomorphic markets are markets which contain identical investment opportunities at identical costs. For completeness, we give a formal definition: a morphism of markets $M_1 = \left((\Omega_1,\mathcal{A}_1,(\mathcal{F}_{1,t})_{t\in\mathcal{T}},\mathbb{P}_1),c_1\right)$ and $M_2 = \left((\Omega_2,\mathcal{A}_2,(\mathcal{F}_{2,t})_{t\in\mathcal{T}},\mathbb{P}_2),c_2\right)$ is a function $\phi:\Omega_1\rightarrow\Omega_2$ satisfying $c_2(X)\geq c_1(X\circ \phi)$ for all $X\in L^0(\Omega_2;\mathbb{R})$ and which is also $\mathcal{F}_{1,t}$-$\mathcal{F}_{2,t}$-measurable for each $t\in\mathcal{T}$. The notion of invariance is now well-defined. See \cite{armstrong2018classifying} for an introduction to the notions of morphisms and invariance as applied to finance.

The fact that the basis assets are not invariant can be interpreted as saying that the choice of basis assets is not financially significant and is merely a convenience for calculations.

To understand a market in an invariant fashion we first note that the Radon-Nikodym density $Q_t$ is invariantly defined by the probability measures $\mathbb{P}$ and $\mathbb{Q}$. However from the point of view of calculations, $Q_t$ is difficult to use directly as it is non-local in time (it is not independent from $Q_0$ for instance). To obtain a basic local invariant we first consider the stochastic logarithm of $Q_t$, given by 
\begin{equation*}
\mathcal{L}(Q_t) \coloneqq \log Q_t + \int_{0}^{t}\frac{1}{2Q_s^2}d\lbrack Q,Q\rbrack_s, 
\end{equation*}
so that $Q_t = \mathcal{E}\left(\mathcal{L}\left(Q_t\right)\right)$ where $\mathcal{E}(\cdot)$ denote the usual Dol\'eans-Dade exponential, $\mathcal{E}(X)_t\coloneqq \exp\left(X_t - \frac{1}{2}\lbrack X,X\rbrack_t\right)$. $d^p\left(\mathcal{L}(Q)\right)_t$ is now the invariant, local quantity we are seeking. This is because taking logarithms and differentiating eliminates the dependency on the prior distribution of $\{Q_t\}$. This in turn is because differences of logarithms of $\{Q_t\}$ are independent of prior distribution. The locality comes from taking the differential rather than a difference. Our next Lemma demonstrates that this quantity can be simply calculated from the coefficients of a diffusion.

We refer to $d^p\left(\mathcal{L}(Q)\right)_t$ as the `market dynamics' or `market differential'. We give the market differential for a diffusion market below. To do so we have to extend our definition of the operation (M) to multiplication of vector-valued random variables with differentials of vector-valued processes.

\begin{definition}
		Let $d^p(X)_t \in\mathds{D}_t\left(\mathbb{R}^m\right)$ and $H\in\mathcal{F}_t\left(\mathbb{R}^m\right)$. We define the \emph{dot product multiplication} of $H$ and $d^p(X)_t$ as
		\begin{equation*}
		H\cdot d^p(X)_t \coloneqq d^p(H\cdot X)_t,
		\end{equation*}
		where the `$\cdot$' on the right-hand side represents the usual dot product between the vector-valued random variable $H$ and the vector-valued process $\{X_s\}_{s\geq 0}$.
\end{definition}

\begin{lemma}[Dynamics of a diffusion market] \label{lemma: dynamics of a diffusion market}
	For the diffusion market in (\ref{eq: diffusion market dynamics}), the market dynamics is $d^p\left(\mathcal{L}(Q)\right)_t = \bm\sigma_t^{-1}(r\bm X_t-\bm \mu_t)\cdot d^p (\bm W)_t$, where we have abbreviated $\bm\sigma(t,\bm X_t) = \bm\sigma_t$ and $\bm\mu(t,\bm X_t) = \bm\mu_t$.
\end{lemma}

\begin{proof}
	For a diffusion market, the Radon--Nikodym derivative is given by \begin{equation*}
	Q_t = \mathcal{E}\left(\int_{0}^{\cdot}\bm\sigma^{-1}_s (r\bm X_t-\bm \mu_s)\cdot d \bm W_s\right)_t,
	\end{equation*}
	see \cite{armstrong2018classifying} for instance. Hence,
	\begin{equation*}
	d^p\left(\mathcal{L}(Q)\right)_t = d^p\left(\int_{0}^{\cdot}\bm\sigma^{-1}_s(r\bm X_t-\bm\mu_s)\cdot d\bm W_s\right)_t = \bm\sigma^{-1}_t(r\bm X_t-\bm \mu_t)\cdot d^p(\bm W)_t,
	\end{equation*}
	where we have used the Fundamental Theorem of Stochastic Calculus.
\end{proof}

\subsection{Instantaneous portfolios and the one-mutual fund separation theorem}

We now define instantaneous portfolios. An instantaneous portfolio at time $t$ consists of the value of the portfolio and its differential at $t$. In a complete market, the portfolio value and its differential have to obey a specific relationship given by the fact that each discounted asset price is a risk-neutral martingale. Let $\{\tilde{V}_t\}$ denote the discounted prices process of an asset. By Girsanov's theorem, we can show that $\left\{\tilde{V}_t + \left\lbrack \tilde{V}, \mathcal{L}(Q)\right\rbrack_t\right\}$ is a $\mathbb{P}$-martingale. By the chain rule
and the Fundamental Theorem of Stochastic Calculus we find that
\begin{equation*}
d^p\left(\tilde{V} + \left\lbrack \tilde{V}, \mathcal{L}(Q)\right\rbrack\right)_t = -re^{-rt}V_t dt + e^{-rt} d^p(V)_t + e^{-rt} d^p\left\lbrack V,\mathcal{L}(Q)\right\rbrack_t,
\end{equation*}
The conditional mean of the left hand side must be zero in a complete market. This gives us the relationship that a portfolio value and its differential must satisfy to be an instantaneous portfolio.

\begin{definition}[Instantaneous portfolios]
	We define the instantaneous portfolios at time $t$ as the set of all pairs $(X,\eta)\in\mathcal{F}_t\times\mathds{D}^\text{It\^o}_t(L^2)$ such that
	\begin{equation} \label{eq: infinitesimal portfolios condition}
	\langle \eta, dt\rangle + \langle d^p\left(\mathcal{L}(Q)\right)_t\star \eta,dt \rangle = rX.
	\end{equation}
	We denote the set of all instantaneous portfolios at time $t$ by $\mathcal{P}_t$.
\end{definition}
As the inner product is bilinear, it is straightforward to check that the set of all instantaneous-portfolios at time $t$ forms linear space. Furthermore, this space has co-dimension $1$ with the space $\mathcal{F}_t\times \mathds{D}^\text{It\^o}_t(L^2)$.

\begin{lemma}
	$(\mathcal{P}_t,+)$ is an $\mathcal{F}_t$-module where addition and $\mathcal{F}_t$-multiplication are defined component-wise.
\end{lemma}

\begin{definition}[Canonical portfolio market]
	We define the canonical portfolio market, $(\mathcal{P}_t,\mathcal{R},\mathcal{C},p)$, as the space of instantaneous portfolios $\mathcal{P}_t$ together with the following maps which respectively corresponds to measure of the risk, cost and payoff of instantaneous portfolios:
	\begin{itemize}
		\item A symmetric bilinear form that is almost-surely positive semi-definite 
		\begin{equation*}
		\mathcal{R}:\mathcal{P}_t\times \mathcal{P}_t \ni \left((X,\eta),(\tilde{X},\tilde{\eta})\right) \mapsto q(\eta,\tilde{\eta})\in\mathcal{F}_t,
		\end{equation*}
		\item Two linear functionals $\mathcal{C}:\mathcal{P}_t\ni (X,\eta)\mapsto X \in \mathcal{F}_t$ and $p:\mathcal{P}_t\ni(X,\eta)\mapsto X + \mathbb{E}_t[\eta]\in\mathcal{F}_t$,
	\end{itemize}
where $\mathbb{E}_t[\cdot]$ and $q(\cdot,\cdot)$ are defined in Definition \ref{def: conditional expectation of an ito differential} and Lemma \ref{lemma: inner product on space of ito differentials} respectively.
\end{definition}

Let us recall the following definitions from \cite{Armstrong_2018}.
\begin{definition}[Expected return and relative risk \cite{Armstrong_2018}]
Let $\pi\in\mathcal{P}_t$ and assume $\mathbb{P}\left(\mathcal{C}(\pi)= 0\right) = 0$. The expected return and the relative risk of $\pi$ are respectively defined by $\text{ER}(\pi)\coloneqq \frac{p(\pi)-\mathcal{C}(\pi)}{\mathcal{C}(\pi)}$ and	$\text{RR}(\pi) \coloneqq \frac{\sqrt{\mathcal{R}(\pi,\pi)}}{\mathcal{C}(\pi)}$.
\end{definition}

\begin{definition}[Risk-free portfolios \cite{Armstrong_2018}]
	An instantaneous portfolio $\pi\in\mathcal{P}_t$ is said to be risk-free if $\mathcal{R}(\pi,\pi)=0$ almost-surely.
\end{definition}
By Lemma 1.6 of \cite{Armstrong_2018}, it is equivalent to say that a risk-free portfolio is one for which $\mathcal{R}(\pi,\tilde{\pi})=0$ almost-surely for all $\tilde{\pi}\in\ker(\mathbb{E}_t)$.

\begin{theorem}[Instantaneous one-mutual fund separation theorem] \label{theorem: one mutual fund}
	In the canonical portfolio market, the set of risk-minimising portfolios is an $\mathcal{F}_t$-submodule of $\mathcal{P}_t$ of dimension $2$ and contains risk-free portfolios. For any feasible payoff and cost, there is an associated risk-minimising portfolio. This is called the one-mutual fund separation theorem because the space of risk-minimising portfolios over all $\mathcal{P}_t$ can be spanned by one risk-free portfolio and one portfolio that is risk-minimising amongst all those portfolios with zero expected return.
\end{theorem}

\begin{proof}
	As we are in the framework of the Markowitz Category, the proof is a direct application of Theorem $2.3$ of \cite{Armstrong_2018}.
\end{proof}

\begin{lemma}
	There is a unique risk-free instantaneous portfolio with cost $X$, namely $(X,rX dt)$. We refer to it as the risk-free portfolio with cost $X$ and write it $\pi_X^0$.
\end{lemma}

\begin{proof}
	All risk-free portfolios have the form $(X,\lambda  dt)$. Plugging this in equation (\ref{eq: infinitesimal portfolios condition}) we find that $\lambda = rX$.
\end{proof}

\begin{lemma}[Instantaneous market portfolio] \label{lemma: instantaneous market portfolio}
	There is a unique risk-minimising portfolio amongst all those portfolios with zero expected return and cost $X$. This is
	\begin{equation*}
\left(X,\frac{rX}{\left\langle d^p\left(\mathcal{L}(Q)\right)_t,d^p\left(\mathcal{L}(Q)\right)_t\right\rangle} d^p\left(\mathcal{L}(Q)\right)_t\right).
	\end{equation*}
	We refer to it as the instantaneous market portfolio with cost $X$ write $\pi_X^\text{market}$.
\end{lemma}

\begin{proof}
	Let $(X,\eta)\in\mathcal{P}_t$ have zero expected return. Then $\mathbb{E}_t[\eta]=0$ and so $X= \frac{q\left(\eta,d^p\left(\mathcal{L}(Q)\right)_t\right)}{r}$. Note that $q(\cdot ,\cdot)$ is a pathwise inner product when restricted to portfolios with zero expected return. Hence, by the Cauchy-Schwartz inequality,
	\begin{equation*}
	\frac{r}{\sqrt{q\left(d^p\left(\mathcal{L}(Q)\right)_t,d^p\left(\mathcal{L}(Q)\right)_t\right)}} \leq \frac{r\sqrt{q(\eta,\eta)}}{q\left(\eta,d^p\left(\mathcal{L}(Q)\right)_t\right)} = \text{RR}\left(X,\eta\right).
	\end{equation*}
	This lower bound is attained if and only if $\eta$ and $d^p\left(\mathcal{L}(Q)\right)_t$ are linearly dependent which gives the result. Note that
	\begin{equation*}
	    \left\langle d^p\left(\mathcal{L}(Q)\right)_t,d^p\left(\mathcal{L}(Q)\right)_t\right\rangle = q\left( d^p\left(\mathcal{L}(Q)\right)_t,d^p\left(\mathcal{L}(Q)\right)_t\right).
	\end{equation*}
\end{proof}

\begin{corollary}
	For a given cost $X$ and payoff $\mu$ the risk-minimising portfolio in Theorem \ref{theorem: one mutual fund} is $\pi^\star = \frac{\mu}{r} \pi_X^0 + \frac{r-\mu}{r} \pi_X^\text{market}$.
\end{corollary}

\begin{proof}
	By the one-mutual fund theorem, we know that the risk-minimising portfolio has differential $\lambda_0  \eta_X^0 + \lambda \eta_X^\text{market}$ for some $\mathcal{F}_t$-measurable random variables $\lambda_0$ and $\lambda$. But the cost of this portfolio is $\lambda_0 X + \lambda X$. Hence $\lambda_0 + \lambda = 1$. Furthermore,
	$\mu X = \mathbb{E}_t\left[\lambda_0 \eta_X^0+\lambda \eta_X^\text{market}\right]= \lambda_0rX$ which completes the proof.
\end{proof}

\begin{definition}[Absolute market price of risk]
We define the \textit{absolute market price of risk} as the reward-to-risk ratio of the instantaneous market portfolio. Namely,
	\begin{equation*}
	\text{AMPR}_t \coloneqq \left\lvert \frac{\text{ER}(\pi_X^\text{market})-\text{ER}(\pi_X^\text{hedge})}{\text{RR}(\pi_X^\text{market})} \right\rvert.
	\end{equation*}
\end{definition}	
For a diffusion market as in (\ref{eq: diffusion market dynamics}), this definition coincides with the one of $\text{AMPR}_t$ found in \cite{armstrong2018classifying}. Our definition generalises it and provides a financial motivation to the definition in terms of instantaneous portfolios.	
\begin{corollary}
\begin{equation*}
    \text{AMPR}_t = \sqrt{q\left(d^p\left(\mathcal{L}(Q)\right)_t,d^p\left(\mathcal{L}(Q)\right)_t\right)} = \sqrt{\left\langle d^p\left(\mathcal{L}(Q)\right)_t,d^p\left(\mathcal{L}(Q)\right)_t\right\rangle}.
\end{equation*}
\end{corollary}

\begin{theorem}
	The space of invariant portfolios in a market with deterministic $\text{AMPR}_t$ and generated by a $n$-dimensional Brownian motion is a two-dimensional $\mathcal{F}_t$-module spanned by $\pi^0_X$ and $\pi^\text{market}_X$ where $X$ is any non-zero cost.
\end{theorem}
This is a general mutual fund theorem that improves Theorem $3.15$ of \cite{armstrong2018classifying} by identifying the invariant instantaneous portfolios and not just invariant processes. The interpretation is that any instantaneous portfolio choice arising from a convex, invariant financial problem lies in the span of $\pi^0_X$ and $\pi^\text{market}_X$.

\begin{proof}
	Without loss of generality the market is a canonical Bachelier market (by the classification theorem of \cite{armstrong2018classifying}). The risk-free portfolio is given by $\pi^0_X = (X,rX dt)$. By Lemmas \ref{lemma: dynamics of a diffusion market} and \ref{lemma: instantaneous market portfolio}, the market portfolio is given by $\pi^\text{market}_X = \left(X, -\frac{rX}{A(t)}e_1\cdot d^p(\bm W)_t\right)=\left(X,-\frac{rX}{A(t)}d^p(W^1)_t\right)$, where $\{e_1,\dots,e_n\}$ is the standard basis for $\mathbb{R}^n$. As $dt$ and $d^p(W^1)_t$ are invariant, so are these two portfolios. Choose orthonormal portfolios $\pi_2,\dots,\pi_{n}$ such that the $\{\pi_i\}_{i=2}^n$ are orthogonal to both $\pi^0_X$ and $\pi^\text{market}_X$. Any instantaneous portfolio may be written as $\alpha_0 \pi^0_X + \alpha_1 \pi^\text{market}_X + \sum_{i=2}^{n}\alpha_i \pi^i$. The map $(e_1,e_2,\dots,e_n)\rightarrow(e_1,-e_2,\dots,-e_n)$ acting on $\mathbb{R}^n$ induces an isomorphism of the Bachelier market which sends
	\begin{equation*}
	(\alpha_0,\alpha_1,\alpha_2,\dots,\alpha_{n})\mapsto (\alpha_0,\alpha_1,-\alpha_2,\dots,-\alpha_{n})
	\end{equation*}
	
	Any invariant portfolio must be fixed under this mapping. So such portfolios must be a linear combination of $\pi^0_X$ and $\pi_X^\text{market}$.
\end{proof}

\section{It\^o stochastic differentials using convergence in mean} \label{section: stochastic differentials using convergence in mean}
This is close in spirit to the work of Nelson in \cite{nelson1967bm}. The definitions of mean forward derivative is taken from \cite{nelson1967bm} and we use it to define the zero differential. A similar definition is used in \cite{backhoffveraguas2019adapted} to define a metric between probability measures.

We prove a Fundamental Theorem of Calculus is possible, but only for uniformly bounded integrands. Unfortunately, no version of It\^o's lemma can hold for processes that do not have all moments defined. Indeed, our definition of the differential using convergence in expectation relies on the existence of the first two moments of the stochastic process. Suppose $X_t$ has only its first $N$ moments defined, then there is no hope of formulating It\^o's lemma for the function $f(x)=x^{N+1}$.

\begin{definition}[Mean-forward derivative \cite{nelson1967bm}]
Let $X_t$ be a sample continuous process which has finite second moments. Let
	\begin{align*}
	\mu(X)_t &\coloneqq  \lim_{h\rightarrow 0^+}\frac{1}{h}\mathbb{E}\left[X_{t,t+h}\big\lvert \mathcal{F}_t\right]
	\intertext{and}
	\sigma\sigma^\top(X)_t &\coloneqq  \lim_{h\rightarrow 0^+}\frac{1}{h}\mathbb{E}\left[(X_{t,t+h}-h\mu(X)_t)(X_{t,t+h}-h\mu(X)_t)^\top\big\lvert \mathcal{F}_t\right],
	\end{align*}
	when those limits exist in $L^1$ and the mappings $t\mapsto \mu(X)_t$ and $t\mapsto \sigma\sigma^\top(X)_t$ are $L^1$-continuous. $\mu(X)_t$ is known as the \emph{mean forward derivative}.
\end{definition}

In virtue of Theorem 11.9 of \cite{nelson1967bm}, we may now focus only on the case when $\left\{X_t\right\}$ to be an It\^o process of the form $X_t = X_0 + \int_{0}^{t}\mu_sds + \int_{0}^{t}\sigma_s dW_s$ for all $t\in[0,T]$, where $t\mapsto \mu_t$ and $t\mapsto \sigma_t\sigma^\top_t$ are $L^1$-continuous and $\sigma_t\sigma^\top_t$ is invertible for a.e.~$t\geq 0$. Note that in this case, $\mu(X)_t = \mu_t$ and $\sigma\sigma^\top(X)_t = \sigma_t\sigma^\top_t$.
\begin{definition}
	$d^E(X)_t = 0$ if and only if $\mu(X)_t = 0$ and $\sigma\sigma^\top(X)_t = 0$.
\end{definition}

\begin{lemma}\label{lemma: dE implies dp for true martingales}
	$d^E(M)_t = 0 \Rightarrow d^p(M)_t = 0$ for any continuous true martingale $M$.
\end{lemma}

\begin{proof}
	We have $\mathbb{E}\left[M_{t,t+h}M_{t,t+h}^\top\right]=\mathbb{E}\left[\lbrack M\rbrack_{t,t+h}\right]$. By the Markov inequality, this implies $\lbrack M\rbrack_{t,t+h} = o_p(h)$.
\end{proof}	

\begin{remark}
	The same is not true of the finite variation part. Indeed, if $A_t = \pm t$ with probability $\nicefrac{1}{2}$, then $\lim_{h\rightarrow 0^+}\frac{1}{h}\mathbb{E}_0[A_{t,t+h}]=0$ but $TV(A)_{0,h}=h\neq o(h)$.
\end{remark}

\begin{theorem}[Uniqueness of solutions]
	If $d^E(X)_t = 0$ for each $t\in[0,T]$, then $X_t=X_0$ for each $t\in[0,T]$.
\end{theorem}

\begin{proof}
Let $\{X_t\}$ have its usual canonical decomposition. We first show that $A_t=0$ a.s. for each $t\in[0,T]$. Pick $s\in[0,T]$. By hypothesis, we have that
\begin{equation}
\forall\epsilon>0,~\exists\delta>0 \text{ s.t. }\forall0< h<\delta, ~\left\Vert\mathbb{E}[X_{s,s+h}\big\lvert \mathcal{F}_s]\right\Vert_{L^1}\leq\epsilon h \label{condition 1}
\end{equation}

Fix $[t,s]\subseteq [0,T]$. Using a similar covering as in the proof of Lemma \ref{lemma: uniqueness of solutions}, we can find a finite number of sub-intervals $\{(t_i,s_i)\}_{i=1}^m$ which cover $[s,t]$ such that none of them is contained in another and (\ref{condition 1}) is satisfied ($\delta_i = s_i-t_i$ for each $i=1,\dots,m$). Using the triangle inequality for the $L^1$-norm and the tower property of expectation, find that
\begin{align*}
\left\Vert\mathbb{E}[X_{s,t}\big\lvert \mathcal{F}_s]\right\Vert_{L^1} &\leq \sum_{j=0}^{m-1}\left\Vert\mathbb{E}[X_{s_j,s_{j+1}}\big\lvert \mathcal{F}_s]\right\Vert_{L^1} \\
&\leq \sum_{j=0}^{m-1}\left\Vert\mathbb{E}\Big[\mathbb{E}[X_{s_j,s_{j+1}}\lvert \mathcal{F}_{s_j}]\big\lvert \mathcal{F}_s\Big]\right\Vert_{L^1} \\
&\leq \sum_{j=0}^{m-1}\left\Vert\mathbb{E}[X_{s_j,s_{j+1}}\big\lvert \mathcal{F}_{s_j}]\right\Vert_{L^1} \\
&\leq \epsilon\sum_{j=0}^{m-1}(s_{j+1}-s_j) = \epsilon(t-s).
\end{align*}
As $(t-s)$ is fixed and $\epsilon$ can be chosen arbitrarily small, we conclude that $\left\Vert\mathbb{E}_s[X_{s,t}]\right\Vert_{L^1}=0$. This implies $\mathbb{E}_s[X_t]= X_s$. In other words, $\{X_t\}$ is a true martingale and by Lemmas \ref{lemma: uniqueness of solutions} and \ref{lemma: dE implies dp for true martingales} we are done. 
\end{proof}

\begin{theorem}[Stochastic Fundamental Theorem of Calculus]
	Let $\{X_t\}$ be an It\^o process of the form $X_t = X_0 + \int_{0}^{t}\mu_sds + \int_{0}^{t}\sigma_s dW_s$ for all $t\in[0,T]$ where $t\mapsto \mu_t$ and $t\mapsto \sigma_t\sigma^\top_t$ are $L^1$-continuous. Then, for any sample continuous, uniformly bounded and adapted process $\{H_t\}$, we have
	\begin{equation*}
	d^E\Big(\int_{0}^{\cdot}H_s dX_s\Big)_t = H_t d^E(X)_t, 
	\end{equation*}
	for all $t\in[0,T]$.
\end{theorem}

\begin{proof}
	For $s\geq t$, define the process $Y_s \coloneqq  \int_{0}^{s}H_u dX_u - H_tX_s = \int_{0}^{s}H_{t,u} dX_u$. As $H$ is uniformly bounded, we know that the drift and diffusion of $Y$ are $L^1$-continuous in $t$. Hence, $\mu(Y)_t = (H_t-H_t)\mu_t = 0$ and $\sigma\sigma^\top(Y)_t = (H_t-H_t)\sigma_t\sigma_t^\top(H_t-H_t)^\top = 0$.
\end{proof}

\section{It\^o stochastic differentials using almost sure convergence} \label{section: stochastic differentials using almost sure convergence}
\begin{definition}\label{def: differential zero almost sure case}
	Let $\{X_t\}$ be a continuous semimartingale with canonical decomposition $X_t = X_0 + A_t+M_t$ and such that there exist progressively measurable processes $\mu$ and $\sigma$ with $A = \int_{0}^{\cdot}\mu_sds$ and $\lbrack M\rbrack = \int_{0}^{\cdot}\sigma^2_sds$, $dt\otimes d\mathds{P}$-almost surely. We will write $d^\text{a.s.}(X)_t=0$ if and only if
	\begin{equation*}
	\lim_{h\rightarrow 0^+}\frac{1}{h}\int_{t}^{t+h}\lVert \mu_s\rVert ds = 0 \text{ and }\lim_{h\rightarrow 0^+}\frac{1}{h}\int_{t}^{t+h} \sigma_s\sigma_s^\top ds = 0,
	\end{equation*} 
	where the limit is taken in the $\mathds{P}$-almost sure sense.
\end{definition}

\begin{lemma}
	Let $\{X_t\}$ be a continuous semimartingale as in Definition \ref{def: differential zero almost sure case}. Then $d^\text{a.s.}(X)_t = 0 \Rightarrow d^p(X)_t=0$.
\end{lemma}

\begin{proof}
As almost-sure convergence implies convergence in probability, we have that $TV(A)_{t,t+h}\leq \int_{t}^{t}\lVert \mu_s\rVert ds=o_p(h)$. Similarly, $\lbrack M\rbrack_{t,t+h} \leq \int_{t}^{t+h}\sigma_s\sigma^\top_s ds = o_p(h)$.
\end{proof}

\begin{corollary}
	Let $\{X_t\}$ be a continuous semimartingale of the form above. If $d^\text{a.s.}(X)_t = 0$ for each $t\in[0,T]$, then $X=X_0$, $dt\otimes d\mathds{P}$-almost surely.
\end{corollary}

\begin{proof}
	The result follows Lemma \ref{lemma: uniqueness of solutions}.
\end{proof}

While we do not prove a Fundamental Theorem of Calculus, instead we get an analogue to the Lebesgue differentiation theorem,

\begin{lemma}
	Let $\{X_t\}$ be of the form above. For a sample continuous, adapted process $\{H_t\}$, we have that, for almost all $\omega\in\Omega$,
	\begin{equation*}
	d^\text{a.s.}\Big(\int_{0}^{\cdot}H_s(\omega) dX_s(\omega)\Big)_t = H_t(\omega) d^\text{a.s.}\big(X(\omega)\big)_t, 
	\end{equation*}
	for almost all $t\in[0,T]$.
\end{lemma}

\begin{proof}
	This follows from the Lebesgue differentiation theorem applied pathwise.
\end{proof}

\appendix
\section{Proof of results} \label{appendix: proof of results}
We could not find proofs of the next twoo results in the literature, but do not claim they are new.
\begin{lemma}[Results using total and quadratic variation do not depend on choice of norm in their definition] \label{lemma: tv and qv definitions do not depend on choice of norm}
	Let $X$ be a continuous semimartingale with canonical decomposition $X_t = X_0 + A_t + M_t$ be a continuous semimartingale on $\mathbb{R}^m$. Suppose that $q_1$ and $q_2$ are norms on $\mathbb{R}^m$. Let $TV(A)_{t,t+h;q_i}$ and $\lbrack X\rbrack_{t,t+h;q_i}$ denote the total and quadratic variation processes over the interval $[t,t+h]$ where the definition uses the norm $q_i$, $i=1,2$. Then we have that
	\begin{equation*}
		TV(A)_{t,t+h;q_1} + \lbrack X\rbrack_{t,t+h;q_1} = \mathcal{O}_p(h) \iff TV(A)_{t,t+h;q_2} + \lbrack X\rbrack_{t,t+h;q_2} = \mathcal{O}_p(h).
	\end{equation*} 
	Moreover, 
	\begin{equation*}
		TV(A)_{t,t+h;q_1} + \lbrack X\rbrack_{t,t+h;q_1} = o_p(h) \iff TV(A)_{t,t+h;q_2} + \lbrack X\rbrack_{t,t+h;q_2} = o_p(h).
	\end{equation*} 
\end{lemma}
\begin{proof}
	As $q_1$ and $q_2$ are equivalent on $\mathbb{R}^m$, there exist constants $a,b>0$ such that for each $\bm x \in\mathbb{R}^m$, $a\cdot q_1(\bm x) \leq q_2(\bm x) \leq b\cdot q_1(\bm x)$. By the definition of total and quadratic variations, we deduce that
	\begin{equation*}
		a\cdot TV(A)_{t,t+h;q_1} \leq TV(A)_{t,t+h;q_2}\leq b\cdot TV(A)_{t,t+h;q_1}
	\end{equation*}
	and
	\begin{equation*}
		a^2\cdot \lbrack X\rbrack_{t,t+h;q_1} \leq \lbrack X \rbrack_{t,t+h;q_2} \leq b^2 \cdot \lbrack X\rbrack_{t,t+h;q_1}.
	\end{equation*}
	By the properties of $o_p(\cdot)$ and $\mathcal{O}_p(\cdot)$, we deduce the results.
\end{proof}

\begin{lemma}[All results are true if and only if they hold component-wise]\label{lemma: all results are true iff they hold component wise}
	Let $X:[0,\infty)\times\Omega \rightarrow \mathbb{R}^m$ be a continuous semimartingale. Then
	\begin{equation*}
		X\in\mathcal{S}^\text{CBP}_t(\mathbb{R}^m) \iff X^{(i)}\in\mathcal{S}^\text{CBP}_t(\mathbb{R}),~\forall i=1,\dots,m.
	\end{equation*}
	Moreover, $d^p(X)_t = 0$ on $\mathbb{R}^m$ if and only if $d^p\left(X^i\right)_t = 0$ on $\mathbb{R}$, for all $i=1,\dots,m$.
\end{lemma}

\begin{proof}
	The proof essentially relies on the following inequality between $p$-norms on $\mathbb{R}^m$: for $\bm x\in\mathbb{R}^m$,
	\begin{equation*}
		\lVert \bm x\rVert_2 \leq \lVert \bm x\rVert_1 \leq \sqrt{m}\cdot \lVert \bm x\rVert_2.
	\end{equation*}
	This implies that for any $\mathbb{R}^m$-valued continuous semimartingale $X$ with decomposition $X=X_0 + A+M$, we have
	\begin{equation*}
		TV(A)_{t,t+h} \leq \sum_{i=1}^{m}TV(A^i)_{t,t+h} \leq \sqrt{m}\cdot \sum_{i=1}^{m}TV(A)_{t,t+h}.
	\end{equation*}
	Moreover, $\lbrack X\rbrack_{t,t+h} = \sum_{i=1}^{m}\lbrack X^i\rbrack_{t,t+h}$, when these quantities exist. By the properties of $o_p(\cdot)$ and $\mathcal{O}_p(\cdot)$ we conclude the result. 
\end{proof}

\begin{proof}[Proof of Example \ref{ex: ito processes have CBP}]
	Suppose first that $\mu$ is continuous with probability one at $t$. By Lemma \ref{lemma: sup of continuous process is op1}
	\begin{multline*}
		\sup_{s\in(t,t+h)}\left\lvert \int_{t}^{t+h}\mu_u du \right\rvert \leq \sup_{s\in(t,t+h)}\int_{t}^{s}\lvert \mu_u\rvert du \\
		\leq \sup_{s\in(t,t+h)}\left\{\sup_{u\in(t,s)}\lvert \mu_u\rvert (s-t)\right\} \leq \sup_{u\in(t,t+h)}\lvert\mu_u\rvert\cdot h = \mathcal{O}_p(1_h)\cdot h = \mathcal{O}_p(h).
	\end{multline*}
Similarly, if $\sigma$ is continuous with probability one at $t$ then
\begin{equation*}
	\left\lbrack \int_{t}^{\cdot}\sigma_udW_u \right\rbrack_{t,t+h} = \int_{t}^{t+h}\sigma^2_udu \leq \sup_{u\in(t,t+h)}\sigma_u^2\cdot h = \mathcal{O}_p(h).
\end{equation*}
Hence by Lemma \ref{lemma: CBP is equivalent to vbp for local martingales} we deduce that $\sup_{s\in(t,t+h)}\left( \int_{t}^{s}\sigma_udW_u \right)^2 = \mathcal{O}_p(h)$.

Suppose instead that $\mu$ is $L^1$-continuous at $t$. Then for each $\epsilon>0$, there exists some $\eta=\eta(\epsilon)>0$ such that $\sup_{u\in(t,t+\eta)}\mathbb{E}\left[\lvert \mu_{t,u}\rvert\right]<\epsilon$. By the triangle inequality this means that
\begin{equation*}
	\sup_{u\in(t,t+\eta)}\mathbb{E}\left[\lvert \mu_u\rvert\right] < \epsilon + \mathbb{E}\left[\lvert \mu_t\rvert\right].
\end{equation*}
Using Markov's inequality, this means that, for all $h\in(0,\eta)$,
\begin{multline*}
	\mathds{P}\left(\sup_{s\in(t,t+h)}\left\lvert \int_{t}^{s}\mu_u du \right\rvert\geq Kh\right) \leq \mathds{P}\left( \int_{t}^{t+h}\lvert \mu_u\rvert du \geq Kh\right) \\
	\leq \frac{1}{Kh}\mathbb{E}\left[ \int_{t}^{t+h}\lvert\mu_u\rvert du\right] 
	\leq \frac{1}{K}\sup_{u\in(t,t+h)}\mathbb{E}\left[\lvert \mu_u\rvert\right]<\frac{\epsilon+\mathbb{E}\left[\lvert \mu_t\rvert\right]}{K}.
\end{multline*}
Picking $K = 1 + \frac{\mathbb{E}\left[\lvert \mu_t\rvert\right]}{\epsilon}$ makes the right-hand side less than $\epsilon$ for any $h\in(0,\eta)$ which proves that $\sup_{s\in(t,t+h)}\left\lvert\int_{t}^{s}\mu_u du \right\rvert = \mathcal{O}_p(h)$ as required. 

Suppose that $\sigma$ were $L^2$-continuous at $t$. Then for any $\epsilon>0$, there exists $\eta=\eta(\epsilon)>0$ such that $\sup_{u\in(t,t+\eta)}\mathbb{E}\left[\left\lvert \sigma_{t,u}\right\rvert^2\right] < \epsilon$. By the triangle inequality for $L^2$ norms, 
\begin{equation*}
	\sup_{u\in(t,t+\eta)}\mathbb{E}\left[ \sigma_u^2\right]^{\nicefrac{1}{2}} < \epsilon^{\nicefrac{1}{2}}+\mathbb{E}\left[\sigma_t^2\right]^{\nicefrac{1}{2}}.
\end{equation*} 
By Markov's inequality and the It\^o isometry, this means that for all $h\in(0,\eta)$,
\begin{multline*}
	\mathds{P}\left(\left\lbrack\int_{t}^{\cdot}\sigma_u dW_u\right\rbrack_{t,t+h}\geq Kh\right) \leq \frac{1}{Kh}\mathbb{E}\left[\int_{t}^{t+h}\sigma^2_u du\right] \\
	\leq \frac{1}{K}\sup_{u\in(t,t+h)}\mathbb{E}\left[\sigma^2_u\right]< \frac{1}{K}\left(\epsilon^{\nicefrac{1}{2}}+\mathbb{E}\left[\sigma^2_t\right]^{\nicefrac{1}{2}}\right)^2.
\end{multline*}
Picking $K = \nicefrac{\epsilon}{\left(\epsilon^{\nicefrac{1}{2}}+\mathbb{E}\left[\sigma^2_t\right]^{\nicefrac{1}{2}}\right)^2}$ makes the right-hand side less than $\epsilon$ for all $h\in(0,\eta)$, which proves that $\left\lbrack\int_{t}^{\cdot}\sigma_u dW_u\right\rbrack_{t,t+h} = \mathcal{O}_p(h)$. By Lemma \ref{lemma: CBP is equivalent to vbp for local martingales} this implies that $\sup_{s\in(t,t+h)}\left(\int_{t}^{s}\sigma_udW_u\right)^2 = \mathcal{O}_p(h)$ as required.
\end{proof}

\section{Useful pathological examples} \label{appendix: useful examples}

We give an example of a continuous function $f$ of bounded variation which satisfies that $\frac{f(h)}{h}\rightarrow 0$ and yet $\frac{TV(f)_{h}}{h}$ is unbounded as $h\rightarrow 0^+$.

\begin{example} \label{ex: pathological example}
	Choose some $(\alpha,\beta)$ in the set
	\begin{equation*}
		\left\{ (x,y)\in\mathbb{R}^2: y>0,~2x+y < -1 < x + y \right\}.
	\end{equation*}
	For instance, $(\alpha,\beta)=(-2,2)$ will do. Construct $f:[0,\infty)\rightarrow [0,\infty)$ as follows:
	\begin{enumerate}[label=(\roman*), itemsep=0pt, topsep=0pt]
		\item $f(0)=0$,
		\item $0 \leq f(t)\leq t^2$, for all $0\leq t\leq 1$,
		\item $f(t) = 0$, for all $t\geq 1$,
		\item The set $\{x_n\}_{n\geq 1}\coloneqq \{n^\alpha\}_{n\geq 1}$ is a set of roots with $x_{n+1} < x_n$, $\forall n\geq 1$,
		\item On $(x_{n+1},x_n)$, there are $\omega_n\coloneqq \lfloor n^\beta\rfloor$ points at which $f(t)=t^2$ and $(\omega_n-1)$ roots in the open interval.
	\end{enumerate}
	It is clear that $0\leq \frac{f(h)}{h}\leq h$, $\forall h>0$. Hence $\frac{f(h)}{h}\rightarrow 0$ as $h\rightarrow 0^+$. It remains to show that $f$ is of bounded variation and that $\frac{TV(f)_h}{h}$ is unbounded as $h\rightarrow 0$. Over any of the intervals $(x_{n+1},x_n)$, we have $2\omega_n x^2_{n+1} < TV(f)_{x_{n+1},x_n} < 2\omega_n x^2_n$.
	Hence,
	\begin{equation*}
		TV(f)_{1} < 2\sum_{m\geq 1} \omega_m x^2_m \text{, and }
		\frac{1}{x_n}TV(f)_{x_n} > \frac{2}{x_n}\sum_{m\geq n}\omega_m x^2_{n+1}. 
	\end{equation*}
	By definition of the floor function, $\lfloor n^\beta\rfloor \leq n^\beta$, so that $TV(f)_{1} < 2\sum_{m\geq 1}m^{2\alpha+\beta}$.
	The right-hand side converges to a finite value since $2\alpha + \beta < -1$. This shows that $f$ is of bounded variation. On the other hand, $\beta>0$, which implies that
	\begin{equation*}
		(m+1)^\beta > m^\beta \text{ and } (m+1)^\beta - \lfloor (m+1)^\beta \rfloor \leq 1.
	\end{equation*}
	This implies that $\lfloor (m+1)^\beta \rfloor \geq (m+1)^\beta - 1 > m^\beta -1$ and thus
	\begin{equation*}
		\frac{1}{x_n}TV(f)_{x_n} > \frac{2}{n^\alpha}\sum_{m\geq n} m^{2\alpha} \lfloor (m+1)\rfloor^\beta > \frac{2}{n^\alpha}\sum_{m\geq n}(m^{2\alpha+\beta}-m^{2\alpha}).
	\end{equation*}
	By looking at the graphs of the functions $y=x^{2\alpha+\beta}$ and $y=x^{2\alpha}$, we can bound the terms above by integrals.
	\begin{align*}
		\sum_{m\geq n}m^{2\alpha+\beta} &> \int_{n}^{\infty}x^{2\alpha+\beta}dx = -\frac{1}{2\alpha+\beta+1}n^{2\alpha+\beta+1}, \intertext{and}
		\sum_{m\geq n}m^{2\alpha} &< \int_{n-1}^{\infty}x^{2\alpha} dx = -\frac{1}{2\alpha+1}(n-1)^{2\alpha+1}.
	\end{align*}
	Hence,
	\begin{align*}
		\frac{1}{x_n}TV(f)_{x_n} &> \frac{2}{n^\alpha}\left(-\frac{1}{2\alpha+\beta+1}n^{2\alpha+\beta+1} + \frac{1}{2\alpha+1}(n-1)^{2\alpha+1}\right) \\
		&= -\frac{2}{2\alpha+\beta+1}n^{\alpha+\beta+1} + \frac{1}{2\alpha+1}\left(\frac{n-1}{n}\right)^\alpha (n-1)^{\alpha+1}.
	\end{align*}
	As $2\alpha+\beta+1<0<\alpha+\beta+1$, the right hand side diverges as $n\rightarrow \infty$. Hence $\frac{TV(f)_h}{h}$ is unbounded as $h\rightarrow 0^+$.
\end{example}

\printbibliography

\end{document}